\documentclass[11 pt]{amsart}
\usepackage{fullpage}
\usepackage{color}
\usepackage{amssymb, amsmath, amsthm}
\usepackage{amsrefs}

\usepackage{graphicx}
\usepackage{pinlabel}
\usepackage{hyperref}

\newtheorem{theorem}{Theorem}[section]
\newtheorem*{theorem*}{Theorem}
\newtheorem{corollary}[theorem]{Corollary}

\newtheorem{lemma}[theorem]{Lemma}
\newtheorem{proposition}[theorem]{Proposition}
\newtheorem*{Additivity Theorem}{Theorem \ref{Main Theorem}}
\newtheorem*{MSS}{Theorem \ref{MSS generalization}}
\newtheorem*{Morimotogen}{Theorem \ref{Morimoto gen}}

\theoremstyle{definition}
\newtheorem{definition}[theorem]{Definition}
\newtheorem{example}[theorem]{Example}
\newtheorem{remark}[theorem]{Remark}

\newcommand{\N}{\mathbb{N}}
\newcommand{\Z}{\mathbb{Z}}
\newcommand{\R}{\mathbb{R}}
\newcommand{\M}{\mathbb{M}}

\newcommand{\co}{\mskip0.5 mu\colon\thinspace} 
\newcommand{\nil}{\varnothing}

\newcommand{\up}{\uparrow}
\newcommand{\down}{\downarrow}
\newcommand{\wihat}[1]{\widehat{#1}}

\newcommand{\defn}[1]{\textbf{#1}}
\renewcommand{\H}{\mathbb{H}}

\newcommand{\more}{\to}

\renewcommand{\S}{\mathbb{S}}

\newcommand{\netchi}{\operatorname{net}\chi} 
\newcommand{\net}{\operatorname{net}}
\newcommand{\netextent}{\operatorname{netext}}
\newcommand{\width}{\operatorname{w}}
\newcommand{\extent}{\operatorname{ext}}
\newcommand{\cpt}{\sqsubset}

\newcommand{\boundary}{\partial}
\newcommand{\bdd}{\boundary}
\newcommand{\mc}[1]{\mathcal{#1}}

\begin{document}

\title{Additive invariants for knots, links and graphs in 3-manifolds}
   \author{Scott A. Taylor, Maggy Tomova}
   
   \begin{abstract}
 We define two new families of invariants for (3-manifold, graph) pairs which detect the unknot and are additive under connected sum of pairs and (-1/2) additive under trivalent vertex sum of pairs. The first of these families is closely related to both bridge number and tunnel number. The second of these families is a variation and generalization of Gabai's width for knots in the 3-sphere. We give applications to the tunnel number and higher genus bridge number of connected sums of knots.
 \end{abstract}
 \maketitle  
 
 \section{Introduction}

Two of the most basic questions concerning any knot invariant are: ``Does it detect the unknot?'' and ``Is it additive under connected sum?'' Among the classical topologically-defined invariants, Seifert genus and bridge number are both well-known for their ``yes'' answers to both questions. Other invariants such as tunnel number and Gabai width, although they both detect the unknot, have more complicated stories when it comes to additivity. In this paper, we define, for almost any graph in almost any 3-manifold, two new families of invariants which both detect the unknot in $S^3$ and are additive under connected sum. For graphs, they also satisfy a certain type of additivity under trivalent vertex sum. In this introduction, we give a brief overview of the definition of the invariants (leaving the technical details until later in the paper), state our main results, and discuss the connection between our invariants and the classical invariants of bridge number, tunnel number, and Gabai width. This work relies on our previous paper \cite{PartI}. A thorough overview of all results we will be needing is provided in Section \ref{sec:prel}.

\subsection{Background}

Knot invariants (and 3--manifold invariants) are a fundamental tool used to distinguish knots (respectively, 3--manifolds) and organize them into useful families. One of the first knot invariants to be introduced was \defn{bridge number} \cite{Schubert}; it is the minimal number of maxima in a diagram of a knot. Bridge number is a beautiful invariant and is connected to a number of important concepts in knot theory (see, for example, \cite{Milnor, Yokota}.) Two of bridge number's particularly useful properties are that it detects the unknot and that it is additive under connected sum\footnote{Technically, it is the quantity that is one less than bridge number which is additive under connected sum.} \cite{Schubert, Schultens}. A \defn{bridge sphere} for a knot $K \subset S^3$ is a sphere transverse to $K$ separating the maxima of $K$ from the minima of $K$ (where the maxima and minima are defined using a height function $h\co S^3 \to K$). A ``higher genus'' version of bridge number was defined by Doll  and is applicable to knots in any compact, orientable 3--manifold. Higher genus bridge number neither detects the unknot nor is additive under connected sum \cite{Doll}.

Heegaard genus is a 3--manifold invariant which is analogous to bridge number. A \defn{Heegaard surface} for a closed, orientable 3--manifold $M$ is a surface separating the 3--manifold into two handlebodies. If we take the double branched cover of a knot $K \subset S^3$, a bridge sphere for $K$ lifts to a Heegaard surface for the resulting 3--manfiold. We can define Heegaard surfaces for 3--manifolds with boundary by considering compressionbodies in place of handlebodies (see \cite{Scharlemann-Survey}; this is also explained more below.) The \defn{Heegaard genus} of a 3--manifold (possibly with boundary) $M$, denoted $g(M)$, is the smallest possible genus of a Heegaard surface in $M$. A closed, orientable 3--manifold $M$ has $g(M) = 0$ if and only if $M = S^3$. Like bridge number, Heegaard genus is additive under connected sum \cite{Jaco}. There is a vast literature on Heegaard surfaces and their usefulness is now well-established.

Tunnel number is a somewhat more recent knot invariant defined by Clark in 1980 \cite{Clark}. The \defn{tunnel number} $t(K)$ of a knot $K \subset S^3$ equals the minimal number of arcs which need to be added to the knot so that the exterior is a handlebody. While tunnel number may at first appear unnatural, it is closely connected to Heegaard genus. Indeed,  $t(K)$ is one less than the Heegaard genus of the exterior of $K$. Largely because of this connection, tunnel number has been extensively studied. Unlike bridge number and Heegaard genus, however, tunnel number behaves erratically under connect sum (see, for example, \cite{Morimoto2, KR1, MSY, MS}). In this paper we seek to correct this by defining a knot invariant (called ``net extent'') which is intimately related to Heegaard genus, bridge number,  higher genus bridge number and tunnel number. Unlike the latter two invariants, however, it both detects the unknot and is additive under connect sum. 

The width of a knot was originally defined by Gabai \cite{G3} as a tool to prove property R. Various other applications of width quickly emerged \cite{GL, Thompson-recog}.  Because of its utility it began to be studied it its own right, the main question being its additivity under connect sum\footnote{Technically, it is the quantity that is two less than Gabai width which was thought to possibly be additive under connected sum. Gabai width does detect the unknot. }  \cite{RS, SchTh-width, BT1}. Eventually, Blair and Tomova \cite{BT} proved that there exist families of knots for which Gabai width is not additive. We introduce a second invariant, also called width, applicable to most any knot in many 3--manifolds. It is a slight variation of Gabai width when restricted to spheres transverse to knots in $S^3$ but \emph{is} additive under connect sum. Section \ref{comparison} below considers the relationship between our width and Gabai width. 

\subsection{The invariants}
For our purposes, a (3-manifold, graph) pair $(M,T)$ consists of a compact orientable 3-manifold $M$ (possibly with boundary) and a properly embedded graph $T \subset M$. 
 
\emph{Running Assumption:} $T$ has no vertices of valence 2 and no component of $\boundary M$ is a sphere intersecting $T$ in precisely two or fewer points.

If $X$ is a topological space (typically a manifold or CW-complex), we let $|X|$ denote the number of connected components of $X$ and write $Y \cpt X$ to mean that $Y$ is a connected component of $X$. If $S \subset$ M is a properly embedded surface, we write $S \subset (M, T )$ to mean that $S$ is transverse to $T$. The notation $(M,T)\setminus S$ refers to the (3--manifold, graph) pair resulting from removing an open regular neighborhhod of $S$ from both $M$ and $T$. We say that $T$ is \defn{irreducible}, if there is no sphere in $M$ intersecting $T$ exactly once. We say that $(M,T)$ is \defn{irreducible} if $T$ is irreducible and if every (tame) sphere in $M\setminus T$ bounds a 3-ball in $M \setminus T$.  A \defn{lens space} is any compact, connected, orientable 3-manifold $M$ such that $M \neq S^1 \times S^2$ and $g(M) = 1$.

Some of the ideas in this paper can be traced to the work of Scharlemann and Thompson on generalized Heegaard splittings \cite{ST-thin}. In a Heegaard splitting, the manifold is split by a single Heegaard surface into two handlebodies (or compressionbodies if the manifold has boundary). In a generalized Heegaard splitting, the manifold is split via two sets of surfaces, thin and thick, into a number of compressionbodies so that the thin surfaces are negative boundaries of the compressionbodies and the thick surfaces are positive boundaries. At first this may seem like an unnecessary complication but it turns out that one can also obtain such a decomposition where all the thin surfaces are incompressible in the manifold and all the thick surfaces are strongly irreducible (i.e., a pair of compressing disks on opposite sides of a thick surface always intersect). This is the fundamental idea behind our work in \cite{PartI} -- we give a decomposition of $(M, T)$ into multiple compressionbodies that intersect $T$ in certain elementary ways. We now make this somewhat more precise. 

As described in \cite{PartI}, a \emph{multiple v.p.-bridge surface} is a closed, usually not connected, orientable surface $\mc{H} \subset (M,T)$ such that $(M,T)\setminus \mc{H}$ is the union of simple-to-understand pieces called \emph{v.p.-compressionbodies}. These multiple v.p.-bridge surfaces are generalizations of bridge surfaces for knots in 3-manifolds, Heegaard surfaces for knot exteriors, and the surfaces arising in Scharlemann-Thompson generalized Heegaard splittings of 3-manifolds. The initials ``v.p.'' stand for ``vertex-punctured'' and indicate that vertices in $T$ are treated in a similar way to boundary components of $M$. We will elaborate on these ideas in Section \ref{sec:prel}.

The components of $\mc{H}$ are partitioned into two sets: the thick surfaces and the thin surfaces. The union of the thick surfaces is denoted $\mc{H}^+$ and the union of the thin surfaces is denoted $\mc{H}^-$ (so $H \cpt \mc{H}^+$ means that $H$ is a thick surface and $F \cpt \mc{H}^-$ means that $F$ is a thin surface.) We will consider multiple v.p.-bridge surfaces $\mc{H}$ which are \emph{reduced} and \emph{oriented}. Roughly speaking, $\mc{H}$ is oriented if the components of $\mc{H}$ are given coherent transverse orientations such that there are no oriented closed loops always intersecting $\mc{H}$ in the same direction and $\mc{H}$ is reduced if there is no ``obvious'' way of simplifying it. See Section \ref{sec:prel} for precise definitions. We let $\H(M,T)$ denote the set of reduced, oriented multiple v.p.-bridge surfaces for $(M,T)$.

If $(M, T)$ is a (3-manifold, graph)-pair and if $S \subset (M,T)$ is a surface, we define the \defn{extent} of $S$ to be
\[
\extent(S) = \frac{|S \cap T| - \chi(S)}{2}.
\] 
If $S$ is connected, this is simply $g(S) - 1 + |S \cap T|/2$ where $g(S)$ is the genus of $S$. Two cases are of particular interest: if $S$ is a minimal bridge sphere for a link $T \subset S^3$, then $\extent(S)$ is  one less than the bridge number $b(T)$ of $T$, and if $S$ is a minimal genus Heegaard surface for the exterior of a knot $T \subset S^3$, then $\extent(S)$ is the tunnel number $t(T)$ of $T$. In both cases, $S$ will meet the requirements for being a v.p.-bridge surface for $(S^3, T)$. 

For a given $\mc{H} \in \H(M,T)$, we define the \defn{net extent} $\netextent(\mc{H})$ and \defn{width} $\width(\mc{H})$ as follows:
\[
\begin{array}{rcl}
\netextent(\mc{H}) &=& \extent(\mc{H}^+) - \extent(\mc{H}^-),\\
\width(\mc{H}) &=&2\left( \sum\limits_{H \cpt \mc{H}^+} \extent^2(H) - \sum\limits_{F \cpt \mc{H}^-} \extent^2(F)\right)
\end{array}.
\]
To each multiple v.p.-bridge surface, we can also associate a number, the \defn{net euler characteristic} of $\mc{H}$ which is simply 
\[
\netchi(\mc{H}) = -\chi(\mc{H}^+) + \chi(\mc{H}^-).\]

Our two families of invariants are then defined as
\[
\begin{array}{rcl}
\netextent_x(M,T) &=& \min\limits_{\mc{H}} \netextent(\mc{H}) \\
\width_x(M,T) &=& \min \limits_{\mc{H}} \width(\mc{H}).
\end{array}
\]
In both cases, the minimum is taken over all  $\mc{H} \in \H(M,T)$ having the property that $\netchi(\mc{H}) \leq x \leq \infty$.  As noted above for a knot $K \subset S^3$, $\netextent_x(S^3, K)$ is related to classical invariants: for any $x \geq -2$, the quantity $\netextent_x(S^3, K)$ is at most $b(K) - 1$ where $b(K)$ is the bridge number of $K$ and, for large enough $x$, $\netextent_x(S^3, K)$ is at most the tunnel number $t(K)$ of $K$. 

The formula for width, on the other hand, is motivated by a well-known formula for Gabai's width invariant \cite{G3} for knots in $S^3$. Indeed, Gabai width for $K \subset S^3$ can be defined as follows. Consider multiple v.p.-bridge surfaces $\mc{H}$ for $(S^3, K)$ with the property that the components of $\mc{H}$ are concentric spheres. Then Gabai width \cite[Lemma 6.2]{SS-width} is the minimum over all such $\mc{H}$ of the quantity
\[
\frac{1}{2}\left(\sum_{H \cpt \mc{H}^+} |H \cap K|^2 - \sum_{F \cpt \mc{H}^-} |F \cap K|^2\right).
\]
Our invariant $\width_{-2}$ for knots in $S^3$ can be seen as a variant of Gabai width, where we generalize the types of surfaces $\mc{H}$ admitted into the sum and adjust the formula to take into account the euler characteristics of the spheres. As mentioned earlier, Gabai width is not additive, while the width defined here is. In Section \ref{comparison} we analyze this phenomenon using the example of non-additivity of Gabai width proven in \cite{BT}.

We prove (Corollary \ref{Cor: non-negativity for surfaces}) that, as long as $x \geq 2g(M) - 2$, where $g(M)$ is the Heegaard genus of $M$, both $\netextent_x(M,T)$ and $\width_x(M,T)$ are non-negative. Indeed, Theorem \ref{Net Extent Detects unknot} implies that if $M$ does not have a lens space or solid torus summand, if $M$ has no non-separating spheres, and if $T$ is connected and non-empty, then $\netextent_x(M,T) = 0$ implies that $M = S^3$ and $T$ is the unknot. A similar result holds for width, although we need to add more hypotheses on $M$ or $x$.

Finally, we make a passing comment on the role of $x$.
\begin{remark}
When working with bridge surfaces or Heegaard surfaces, it is often useful to maintain some control over the euler characteristic of the surfaces. Introducing the parameter $x$ allows us to do that. Observe that, by the definition, both net extent and width are non-increasing as $x$ increases. That is, for all $x \in \Z$,
\[
\begin{array}{rcl}
\netextent_x(M,T) &\geq & \netextent_{x+1}(M,T) \\
\width_x(M,T) & \geq & \width_{x+1}(M,T).
\end{array}
\]
It is easily seen that the values of both $\netextent_x$ and $\width_x$ are integers or half-integers. Thus, the sequences $(\netextent_x(M,T))_{x}$ and $(\width_x(M,T))_{x}$ are eventually constant at $\netextent_\infty(M,T)$ and $\width_\infty(M,T)$.
\end{remark}

\subsection{Additivity}

Suppose that $(\wihat{M}_1,\wihat{T}_1)$ and $(\wihat{M}_2, \wihat{T}_2)$ are disjoint (3-manifold, graph) pairs such that $p_1 \in \wihat{T}_1$ and $p_2 \in \wihat{T}_2$ are either both disjoint from the vertices of $\wihat{T}_1$ and $\wihat{T}_2$ or both are trivalent vertices of $\wihat{T}_1$ and $\wihat{T}_2$. Let $k = 2$ if both are disjoint from the vertices and let $k = 3$ if both are trivalent vertices. We can form a new (3-manifold, graph) pair
\[
(M,T) = (\wihat{M}_1,\wihat{T}_1) \#_k (\wihat{M}_2, \wihat{T}_2)
\]
as follows: Remove an open regular neighborhood of $p_1$ and $p_2$ from $(\wihat{M}_1,\wihat{T}_1)$ and $(\wihat{M}_2, \wihat{T}_2)$ to produce spheres $P_1$ and $P_2$ in the boundaries of the resulting pairs $(M_1, T_1)$ and $(M_2, T_2)$ respectively. The spheres $P_1$ and $P_2$ are both either twice-punctured or thrice-punctured by $T_1$ and $T_2$. Let $(M,T)$ be the result of gluing the (3-manifold, graph)-pairs together by a homeomorphism $P_1 \to P_2$ taking $T_1 \cap P_1$ to $T_2 \cap P_2$. We call the image of $P_1$ (and $P_2$) in $(M,T)$ the \defn{summing sphere}. If $k = 2$, we say that $(M,T)$ is the \defn{connected sum} of $(\wihat{M}_1,\wihat{T}_1)$ and $(\wihat{M}_2, \wihat{T}_2)$; if $k = 3$, then $(M,T)$ is the \defn{trivalent vertex sum} of $(\wihat{M}_1,\wihat{T}_1)$ and $(\wihat{M}_2, \wihat{T}_2)$. We will usually write $(M,T) = (\wihat{M}_1,\wihat{T}_1) \# (\wihat{M}_2, \wihat{T}_2)$ in place of $(M,T) = (\wihat{M}_1,\wihat{T}_1) \#_2 (\wihat{M}_2, \wihat{T}_2)$.

For our purposes, we will say that a pair $(M,T)$ is \defn{trivial} if it is $(S^3, T)$ where $T$ is either an unknot or a trivial $\theta$-graph (i.e. a graph having exactly two vertices and exactly three edges each joining the two vertices which can be isotoped into a Heegaard sphere for $S^3$.) If $(M,T) = (\wihat{M}_1,\wihat{T}_1) \#_k (\wihat{M}_2, \wihat{T}_2)$, then the summing sphere is essential in $(M,T)$ if neither $(\wihat{M}_1,\wihat{T}_1)$ or $(\wihat{M}_2, \wihat{T}_2)$ is trivial.

We say that $(\wihat{M}_1, \wihat{T}_1), \hdots, (\wihat{M}_n, \wihat{T}_n)$ is a \defn{prime decomposition} of $(M,T)$ if  all of the following hold:
\begin{itemize}
\item either $n = 1$ and $(\wihat{M}_1, \wihat{T}_1) = (M,T)$ or $n \geq 2$, and $(M,T)$ is the result of sequentially connect summing and trivalent vertex summing the $(\wihat{M}_i, \wihat{T}_i)$ together (in some order).
\item For all $i$, if $(\wihat{M}_i, \wihat{T}_i)$ is trivial, then $(\wihat{M}_i, \wihat{T}_i) = (S^3, \text{ trivial } \theta\text{-graph })$ and only connected sums are performed on $(\wihat{M}_i, \wihat{T}_i)$.
\item For all $i$, if $P \subset (\wihat{M}_i, \wihat{T}_i)$ is an essential sphere, then either $P \cap \wihat{T}_i = \nil$ or $|P \cap \wihat{T}_i| \geq 4$. 
\end{itemize}

Since we require that the summing be done sequentially, the graph in $M$ dual to the summing spheres is a tree. Under the assumption that no sphere in $M$ is non-separating, that $T$ is a knot, and that $(M,T)$ is non-trivial, then Miyazaki \cite[Theorem 4.1]{Miyazaki} has shown that $(M,T)$ has a unique prime factorization, up to re-ordering. This was extended to the situation where $T$ is a $\theta$-graph by Matveev and Turaev \cite{MT}.
 
Let $\M$ be the set whose elements are irreducible (3-manifold, graph) pairs $(M,T)$ satisfying the running assumption and with $M$ connected, with $T$ non-empty, and where every sphere in $M$ separates.  Let $\M_2 \subset \M$ be the subset where $g(M) \leq 2$ and let $\M_s \subset \M$ be the subset where every closed surface in $M$ separates. We prove:

\begin{Additivity Theorem}[Additivity Theorem]
Let $(M,T) \in \M$ be non-trivial, and let $x$ be any integer with $x \geq 2g(M) - 2$. Then there is a prime factorization of $(M,T)$ into $(\wihat{M}_1, \wihat{T}_1), \hdots, (\wihat{M}_n, \wihat{T}_n)$ so that there exist integers $x_1, \hdots, x_n$, summing to at most $x - 2(n-1)$,  and
\[
\netextent_x(M,T) =  -p_3/2+ \sum_{i = 1}^n \netextent_{x_i}(\wihat{M}_i, \wihat{T}_i).
\]
where $p_3$ is the number of thrice punctured spheres in the decomposition. Furthermore, if $(M,T) \in \M_s$ or if $(M,T) \in \M_2$ and $x \leq 2$, then also
\[
\width_x(M,T) =  -p_3/2+ \sum_{i = 1}^n \width_{x_i}(\wihat{M}_i, \wihat{T}_i).
\]
\end{Additivity Theorem}

The result for width is particularly striking.  As we previously mentioned, for many years, it was an open question as to whether or not Gabai width satisfied an additivity property with respect to connected sum of knots. However, Blair and Tomova \cite{BT} proved that width is not additive. On the other hand, Theorem \ref{Main Theorem} shows that our invariant $w_{-2}$, which is a slightly modified version of Gabai width, \emph{is} additive under connected sum and, even more surprisingly, ``higher genus'' widths are also additive. For more details on the relationship between our width and Blair and Tomova's counterexamples, see Section \ref{comparison}.

\subsection{Applications to classical invariants}
We give several simple applications of our results to knots in 3-manifolds. For the statement, recall that a knot $K$ in a 3-manifold $M$ is \defn{meridionally small} or \defn{m-small} if there is no surface $S \subset (M,K)$ such that $S \cap K \neq \nil$ and $S$ is essential in $(M,K)$ (i.e. is incompressible and not $\boundary$-parallel in the exterior of $K$.) 

We give two short proofs (Theorem \ref{thm:schubert} and Theorem \ref{thm:norwood}) of classical results of Schubert \cite{Schubert} and Norwood \cite{Norwood} showing that 2-bridge knots and tunnel number 1 knots (more generally) are both prime. Scharlemann and Schultens \cite{ScharlemannSchultens-Tunnel} generalized Norwood's result to show that the tunnel number of the connected sum of $n$ knots is at least $n$.  Morimoto \cite{Morimoto} proved a stronger result for m-small knots: the tunnel number of the connected sum of $n$ m-small knots is at least the sum of the tunnel numbers of the factors. We prove a theorem which combines the Scharlemann-Schultens and Morimoto results. Dropping the hats off the summands for convenience, the statement is:

\begin{MSS}
For each $i \in \{1, \hdots, n\}$ let $K_i$ be a knot in a closed, orientable 3-manifold $M_i$ such that every sphere in $M_i$ separates and each $(M_i, K_i)$ is prime. Assume that there is an integer $j \leq n$ so that $K_i$ is m-small if and only if $i \leq j$. Then, letting $(M,K) = (M_1, K_1) \# \cdots \# (M_n, K_n)$, we have:
\[
(n - j) + t(K_1) + \hdots + t(K_j) \leq t(K) \leq (n-1) + \sum t(K_i)
\]
\end{MSS}

Kobayashi-Rieck \cite{KR}  studied the asymptotic properties of tunnel number of the connected sum $nK$ of a knot $K$ with itself $n$ times. As part of that project, they showed that for ``admissible'' m-small knots $K$,
\[
0 \leq \lim\limits_{n \to \infty} \frac{t(nK) - nt(K)}{n-1} <1. 
\]
Along the way, they prove the left-hand inequality holds for each term of the sequence (not just in the limit.) Our Theorem \ref{MSS generalization}, gives another proof that $0 \leq \frac{t(nK) - nt(K)}{n-1}$ for m-small knots.

If $K \subset M$ is a link, a surface $S$ is a \defn{genus $g$ bridge surface} for $K$ if after removing a regular neighborhood of all components of $K$ that are disjoint from $S$, $S$ is a Heegaard surface for the resulting manifold which intersects $K$ transversally and divides $K$ into arcs parallel into $S$. The \defn{genus $g$ bridge number} of $(M,K)$ is the smallest natural number $b_g(K)$ such that there is a genus $g$ bridge surface for $K$. 

\begin{remark}
This is not quite the definition given by Doll \cite{Doll} for higher genus bridge number. He defines $b_g(K)$ only when it is positive. Subsequently, opinions have differed on how to extend the definition to allow $b_g(K) = 0$. Some authors (as we do) declare $b_g(K) = 0$ if and only if $K$ is a core loop for a genus $g$ Heegaard splitting and others if and only if $K$ is isotopic into a genus $g$ Heegaard surface for $M$. 
\end{remark}

A knot $K$ in a 3-manifold $M$ is \defn{small} if $M \setminus K$ contains no closed essential surfaces. By \cite[Theorem 2.0.3]{CGLS}, a small knot in $S^3$ is also m-small, but we will not use that fact. Observe, however, that if $M$ contains a non-separating sphere, then $(M,K)$ is not small and m-small. We show that this higher genus bridge number satisfies a certain super-additivity for small knots, in the following sense:

\begin{theorem}\label{Adding bridge number}
Suppose that $(M_i, K_i)$ are small and m-small for $i \in \{1, \hdots, n\}$. Let $(M,K) = \#_{i =1}^n(M_i, K_i)$ and let $g \geq g(M,K)$. Then there exist $g_i$ such that $\sum g_i \leq g$, $g_i \geq g(M_i, K_i)$ and 
\[
\sum\limits_{i=1} (g_i + b_{g_i}(K) - 1) \leq g + b_g(K) - 1.
\]
\end{theorem}

Restricting to two summands for convenience, we have:

\begin{theorem}
Suppose that for $i = 1,2$ the pair $(M_i, K_i)$ is small and m-small. Let $g_1 = g(M_1)$ and $g_2 = g(M_2)$ and assume that $t(K_i) \geq g_i$ for $i \in \{1,2\}$.  Let $(M,K) = (M_1, K_1) \# (M_2, K_2)$ and let $g = g(M)$. Then:
\[b_g(M,K) = b_{g_1}(M_1, K_1) + b_{g_2}(M_2, K_2) - 1.\]
\end{theorem}
\begin{proof}
By Theorem \ref{Adding bridge number}, given $g \geq g(M_1 \# M_2)$, there exist $g_1$ and $g_2$ such that $g_1 + g_2 \leq g$ and so that
\[
(g_1 + b_{g_1}(K_1) - 1) + (g_2 + b_{g_2}(K_2) - 1) \leq g + b_g(K) - 1.
\]
Since $g(M_1) + g(M_2) = g$, we actually have $g_1 + g_2 = g$. 
Thus,
\[
(b_{g_1}(K_1) - 1) + (b_{g_2}(K_2) - 1) \leq b_g(K) - 1.
\]

By adapting Lemma \cite[Bridge Inequality 1.2]{Doll} to our definitions,  we have
\[
b_{g}(K)\leq \max(b_{g_1}(K_1),1) + \max(b_{g_2}(K_2),1) - 1. 
\]
Since $t(K_i) \geq g_i$, we have:
\[
b_{g}(K)\leq b_{g_1}(K_1) + b_{g_2}(K_2) - 1. 
\]

Thus,
\[
b_g(K) = b_{g_1}(K_1) + b_{g_2}(K_2) - 1
\]
\end{proof}

This solves Doll's Conjecture 1.1 for small knots $K_1$ and $K_2$ and for $g = g(M_1) + g(M_2)$.

Finally, we we consider composite $(g,b)$-knots. A knot $K$ has a $(g,b)$-decomposition if there is a genus $g$ bridge surface intersecting $K$ in $2b$ points. If $b = 0$, this means that $K$ is a core loop of a compressionbody to one side of a genus $g$ Heegaard surface. The knot $K$ is a $(g,b)$-knot if it has a $(g,b)$-decomposition and does not have either a $(g-1,b+1)$ or $(g, b-1)$ decomposition. A (0,2)-knot is also called a 2-bridge knot. Morimoto \cite{Morimoto15} showed that composite $(0,3)$ knots are the sum of two 2-bridge knots and composite $(1,2)$-knots are the connected sum of a 2-bridge knot and a $(1,1)$-knot. We prove a far-reaching generalization of Morimoto's theorems:

\begin{Morimotogen}
Suppose that $K \subset S^3$ is a composite $(g,b)$ knot. Then the number of prime summands is at most $g + b - 1$. If the number of summands is exactly  $m = g+ b - 1$, then at least
\[ \frac{g}{2} + (b-1)\] of the summands have (1,1)-decompositions and at least $(b-1)$ of those are 2-bridge knots.
\end{Morimotogen}

Corollary \ref{MorimotoCor} explains how to obtain Morimoto's theorems from this result.
 
 \section{Preliminaries}\label{sec:prel}

\subsection{Additional Notation}
For a (3-manifold, graph) pair $(M,T)$, recall that we use $M\setminus T$ denote the exterior of $T$ and for a surface $S \subset (M,T)$, the notation $(M,T) \setminus S$ denote the result of removing an open regular neighborhood of $S$ from both $M$ and $T$. All surfaces appearing in this paper are tame, compact, and orientable. A surface $S \subset (M,T)$ is \defn{essential} if $S\setminus T$ is incompressible and not $\boundary$-parallel in $M\setminus T$ and $S$ is not a 2-sphere disjoint from $T$ bounding a 3-ball disjoint from $T$. We use $(C, T_C) \cpt (M,T)\setminus S$ to indicate that $C$ is a component of $M\setminus S$ and $T_C = T \cap C$.

\subsection{Generalizations of compressionbodies}
We begin by generalizing the usual notion of ``compressionbody'' to obtain objects we call ``v.p.-compressionbodies''.  Just as traditional compressionbodies can be cut open along a collection of discs to obtain 3--balls and the product of a surface with an interval, so our ``v.p.-compressionbodies'' can be cut open along generalizations of compressing discs, called ``sc-discs'', to obtain very simple (3--manifold, graph) pairs. Essentially, an ``sc-disc'' is a compressing disc that is allowed to intersect the graph in a single point. For more on why we need this generalization, see \cite{PartI} .

\begin{definition}
Suppose that $S \subset (M,T)$ is a surface and that $D$ is an embedded disc in $M$ such that the following hold:
\begin{enumerate}
\item $\boundary D \subset (S \setminus T)$, the interior of $D$ is disjoint from $S$, and $D$ is transverse to $T$.
\item $|D \cap T| \leq 1$.
\item There is no disc $E \subset S$ such that $\boundary E = \boundary D$ and $E \cup D$ bounds either a 3-ball in $M$ disjoint from $T$ or a 3-ball in $M$ whose intersection with $T$ consists entirely of a single unknotted arc with one endpoint in $E$ and one endpoint in $D$.

\end{enumerate}
Then $D$ is an \defn{sc-disc}. We categorize sc-discs into compressing discs, cut discs, semi-cut discs, and semi-compressing discs. If $|D \cap T| = 0$ and $\boundary D$ does not bound a disc in $S\setminus T$, then $D$ is a \defn{compressing disc}. If $|D \cap T| = 0$ and $\boundary D$ does bound a disc in $S\setminus T$, then $D$ is a \defn{semi-compressing disc}. If $|D \cap T| = 1$ and $\boundary D$ does not bound an unpunctured disc or a once-punctured disc in $S\setminus T$, then $D$ is a \defn{cut disc}. If $|D \cap T| = 1$ and $\boundary D$ does bound an unpunctured disc or a once-punctured disc in $S\setminus T$, then $D$ is a \defn{semi-cut disc}. A \defn{c-disc} is a compressing disc or cut disc. The surface $S \subset (M,T)$ is \defn{c-incompressible} if $S$ does not have a c-disc; it is \defn{c-essential} if it is essential and c-incompressible. If $S$ is separating and there is a pair of disjoint sc-discs on opposite sides of $S$, then $S$ is \defn{sc-weakly reducible}, otherwise it is \defn{sc-strongly irreducible}. 
\end{definition}
 
 \begin{remark}
 A Heegaard surface in a 3-manifold $M$ is \defn{weakly reducible} if it has pair of compressing discs on opposite sides that are disjoint from each other. Casson and Gordon \cite{CG} (see also \cite{ST-thin, Moriah}) showed that weakly reducible Heegaard surfaces often give rise to essential surfaces in the 3-manifold. In \cite{PartI} (see Theorem \ref{Properties Locally Thin} below), we explain how to strengthen these results to  (3--manifold, graph) pairs. Hempel \cite{Hempel} reinterpreted weak reducibility in terms of the curve complex of the Heegaard surface. For Hempel, the ``distance'' of a Heegaard surface is the distance in the curve complex between the disc sets for the 3--manifolds on either side of the Heegaard surface. This definition can be extended to apply to any separating surface which is compressible to both sides. We could then reinterpret our notion of sc-weakly reducible in terms of the distance between disc sets corresponding to the sets of sc-discs on either side of the surface. However, in what follows, we do not need this interpretion.
 \end{remark}
 
We can now define our generalization of traditional compressionbodies. See Figure \ref{Fig:  arc types} for an example.

\begin{definition}
Suppose that $F$ is a closed, connected, orientable surface. We say that $(F \times I, T)$ is a \defn{trivial product compressionbody} if $T$ is isotopic to the (possibly empty) union of vertical arcs. We let $\boundary_\pm (F \times I) = F \times \{\pm 1\}$. If $B$ is a 3--ball and if $T \subset B$ is a (possibly empty) connected, properly embedded, $\boundary$-parallel tree, having at most one interior vertex, then we say that $(B, T)$ is a \defn{trivial ball compressionbody}. We let $\boundary_+ B = \boundary B$ and $\boundary_- B = \nil$.  A \defn{trivial compressionbody} is either a trivial product compressionbody or a trivial ball compressionbody. 

A pair $(C,T)$, with $C$ connected, is a \defn{v.p.-compressionbody} if there is some component, denoted $\bdd_+C$, of $\bdd C$ and a collection of pairwise disjoint sc-discs $\mc{D} \subset (C,T)$ for $\boundary_+ C$  such that the result of $\boundary$-reducing $(C,T)$ using $\mc{D}$ is a union of trivial compressionbodies. The set of sc-discs $\mc{D}$ is called a \defn{complete collection of sc-discs for $(C,T)$}. The set $\bdd C \setminus \bdd_+C$ is denoted by $\bdd_-C$. 

An edge of $T$ disjoint from $\boundary_+ C$ is a \defn{ghost arc}. An edge of $T$ with one endpoint in $\boundary_+ C$  and one in $\boundary_- C$ is a \defn{vertical arc}. A component of $T$ which is an arc having both endpoints on $\boundary_+ C$  is a \defn{bridge arc}. A component of $T$ which is homeomorphic to a circle and is disjoint from $\boundary C$ is called a \defn{core loop}.  A \defn{bridge disc} for $\boundary_+ C$ in $C$ is an embedded disc in $C$ with boundary the union of two arcs $\alpha$ and $\beta$ such that $\alpha \subset \boundary_+ C$ joins distinct points of $\bdd_+C \cap T$ and $\beta$ is the union of edges of $T$. We will only consider bridge discs which are disjoint from the vertices of $T$.
 \end{definition}

Figure \ref{Fig:  arc types} depicts a v.p.-compressionbody $(C,T)$ containing three vertical arcs, one ghost arc, one bridge arc, and one core loop. If $(C,T)$ is a v.p.-compressionbody such that $T$ has no interior vertices, then every component of $T$ is either a vertical arc, ghost arc, bridge arc, or core loop. We will often reduce to this situation by drilling out vertices of $T$ (i.e. removing a regular neighborhood of them so that vertices correspond to new spherical boundary components of the resulting 3-manifold.)

\begin{center}
\begin{figure}[tbh]
\centering
\labellist \small\hair 2pt 
\pinlabel {$\boundary_+ C$} [r] at 1 139
\pinlabel {$\boundary_- C$} [r] at 35 7
\endlabellist
\includegraphics[scale=0.4]{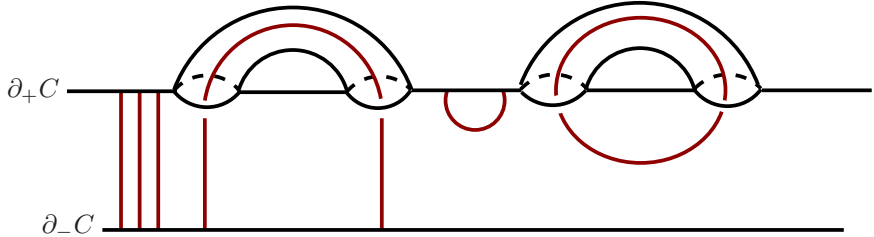}
\caption{A v.p.-compressionbody $(C,T)$. From left to right we have three vertical arcs, one ghost arc, one bridge arc, and one core loop in $T$. This figure was reappropriated from \cite{PartI}.}
 \label{Fig:  arc types}
\end{figure}
\end{center}

\begin{remark}\label{handle remark 1}
Just as in the traditional setting, there is a way of viewing v.p.-compressionbodies in terms of 3--dimensional handle theory. Although this perspective underlies our approach, we do not explicitly rely on it in what follows. Classically, compressionbodies can be constructed by starting with 3--balls and $F \times I$ for a closed surface $F$ (let us call these ``0--handles'') and then attaching 1--handles, each of the form $D^2 \times I$ so that $D^2 \times \boundary I$ lies on the union boundary of the 3--balls with $F \times \{1\}$. There is a similar construction of v.p.-compressionbodies. Consider trivial compressionbodies as 0--handles. There are two kinds of 1--handles. One is the (3--manifold, graph) pair $(D^2 \times I, \nil)$ (i.e. a traditional 3--dimensional 1--handle). The other is the pair $(D^2 \times I, (0,0) \times I)$ (i.e. a traditional 1--handle together with its core.) We then attach 1--handles to 0--handles by gluing $D^2 \times \boundary I$ to the positive boundary of the 0-handles. We also insist that for the first kind of 1--handle the attaching region is disjoint from the graph in the 0--handles and for the second kind of 1--handle, the center of each disc in the attaching region is glued to an endpoint of the graph in the 0-handle and the attaching region is otherwise disjoint from the graph in the 0--handles. 

Equivalently, a v.p-compressionbody can be constructed by starting with a trivial product compressionbody $\boundary_+ C \times I$ and then attaching certain kinds of 2--handles and 3--handles to $\boundary_+ C \times \{0\}$. As with 1--handles, there are two kinds of 2--handles: one kind has no graph in it and the other kind contains the cocore of the 2--handle. 3--handles are equivalent to 0--handles, but with non-empty attaching region.
 \end{remark}

What follows is a key property of v.p.-compressionbodies that we will use on several occasions.

\begin{lemma}[{\cite[Lemma 3.5]{PartI}}]\label{Lem: Invariance}
Suppose that $(C,T)$ is a v.p.-compressionbody such that no component of $\boundary_- C$ is a 2-sphere intersecting $T$ exactly once. The following are true:
\begin{enumerate}
\item\label{it: no sc} $(C,T)$ is a trivial compressionbody if and only if there are no sc-discs for $\boundary_+ C$.
\item\label{it: reducing} If $D$ is an sc-disc for $\boundary_+ C$, then reducing $(C,T)$ using $D$ is the union of v.p.-compressionbodies. Furthermore, there is a complete collection of sc-discs for $(C,T)$ containing $D$.  
\end{enumerate}
\end{lemma}

\subsection{Thick and thin surfaces}
Gabai width \cite{G3} for a knot $K$ is defined using Morse functions $h \co S^3 \to \R$ which restrict to  Morse functions $h|_K$ on the knot and then considering how the maxima and minima of $h|_K$ relate to each other. Inspired by this, Scharlemann and Thompson \cite{ST-thin} defined the width of a closed 3--manifold $M$ by considering handle decompositions of $M$ with a single 0--handle and a single 3-handle and examining how the 1--handles and 2--handles relate to each other. For both Gabai width and Scharlemann-Thompson's width, it is the associated thick and thin surfaces which make the theories especially useful. Hayashi and Shimokawa \cite{HS01} focused attention on the union of these surfaces (which they call a ``multiple Heegaard splitting''). Here is our version; see \cite{PartI} for more detail and motivation.

\begin{definition}
A  \defn{multiple v.p.-bridge surface} for $(M,T)$ is a closed (possibly disconnected) surface $\mc{H} \subset (M,T)$ such that:
\begin{itemize}
\item $\mc{H}$ is the disjoint union of $\mc{H}^-$ and $\mc{H}^+$, each of which is the union of components of $\mc{H}$;
\item $(M,T)\setminus \mc{H}$ is the union of embedded v.p.-compressionbodies $(C_i, T_i)$ with $\mc{H}^- \cup \boundary M= \bigcup \boundary_- C_i$ and $\mc{H}^+ = \bigcup \boundary_+ C_i$;
\item Each component of $\mc{H}$ is adjacent to two distinct v.p.-compressionbodies.
\end{itemize}
The components of $\mc{H}^-$ are called \defn{thin surfaces} and the components of $\mc{H}^+$ are called \defn{thick surfaces}. If $\mc{H}$ is connected, then $\mc{H} = \mc{H}^+$ is called a \defn{v.p.-bridge surface} for $(M,T)$.
\end{definition}

Observe that, for a multiple v.p.-bridge surface $\mc{H}$ of $(M,T)$, each component of $\mc{H}^+$ is a v.p.-bridge surface for the component of $(M,T)\setminus \mc{H}^-$ containing it.

We are usually interested in multiple v.p.-bridge surfaces that have certain additional properties:
\begin{definition}[{for details, see \cite[Section 3.2]{PartI}}]
Suppose that $\mc{H}$ is a multiple v.p.-bridge surface for $(M,T)$. Suppose that each component of $\mc{H}$ is given a transverse orientation so all orientations are consistent on the boundary of each v.p.-compressionbody. (That is, after also giving each component of $\boundary M \cap \boundary_- C$ a transverse orientation, each compressionbody $C$ is an oriented cobordism from $\boundary_- C$ to $\boundary_+ C$ or vice versa.) A \defn{flow line} for $\mc{H}$ is a non-constant oriented path in $M$ always intersecting $\mc{H}$ in the direction of the transverse orientation, transverse to and not disjoint from $\mc{H}$. The multiple v.p.-bridge surface $\mc{H}$ is \defn{oriented} if there are no closed flow lines.

\end{definition}

\begin{remark}
The transverse orientation induces a certain kind of handle decomposition of $(M,T)$, with handles as in Remark \ref{handle remark 1}. Similarly, each thick surface $H \cpt \mc{H}^+$ induces a Morse function on the component $M_0$ of $M \setminus \mc{H}$ containing it. This Morse function can be chosen so that it restricts to a Morse function on a subgraph of $M_0 \cap T$. The subgraph is the union of all components of $M_0 \cap T$ other than the ghost arcs on either side of $H$. The Morse functions corresponding to each component of $M \setminus \mc{H}$ can be pieced together to give a Morse-like function from $M$ to a certain graph, but we do not pursue this line of inquiry here.
\end{remark}

Just as a Heegaard surface for a 3-manifold can be stabilized and thus have higher genus than necessary, so a multiple v.p.-bridge surface may have thick surfaces that are higher genus or have more punctures than necessary. In \cite{PartI}, we defined a collection of destabilizing moves for multiple v.p.-bridge surfaces. These generalize the traditional notions of stabilization and $\boundary$-stabilization of Heegaard splittings of 3--manifolds. The types of destabilization for $H \cpt \mc{H}^+$  are as follows (see \cite{PartI} for precise definitions). All of these are called \defn{generalized destabilizations}.
\begin{itemize}
\item  (Destabilization) Compressing along a certain compressing disc for $H$ having boundary which is non-separating on $H$.
\item (Meridional Destabilization) Compressing along a certain cut disc for $H$ having boundary which is non-separating on $H$.
\item (Boundary Destabilization) Compressing along a certain separating compressing disc for $H$ and discarding a component of the resulting surface.
\item (Meridional Boundary Destabilization) Compressing along a certain separating cut disc for $H$ and discarding a component of the resulting surface.
\item (Ghost Boundary Destabilization) Compressing along a certain separating compressing disc for $H$ and discarding a component of the resulting surface.
\item (Ghost Meridional Boundary Destabilization) Compressing along a certain separating cut disc for $H$ and discarding a component of the resulting surface.
\end{itemize}

Meridional destabilization and meridional boundary destabilization are essentially the same, except that a cut disc plays the role of the compressing disc. A ghost (meridional) boundary destabilization is the same as (meridional) boundary destablization after removing an open regular neighborhood of a certain subgraph of $T$ from $(M,T)$. 

There are times when it is possible to isotope a component of $\mc{H}^+$ across a bridge disc so as to reduce the number of intersections between $\mc{H}$ and $T$ while still producing a multiple v.p.-bridge surface. Two are of special interest (see Figure \ref{Fig:unperturb}). These operations have shown up in other contexts (see, for example \cite{HS01, STo}).
\begin{itemize}
\item (Unperturbing) Isotope $H \cpt \mc{H}^+$ across a bridge disc $D$ which shares a single point of intersection with a bridge disc on the opposite side of $H$. The result of this isotopy is that the number of intersections of $H$ and $T$ is reduced by two.
\item (Undoing a removable arc) Isotope $H \cpt \mc{H}^+$ across a bridge disc $D$ that has a single point of intersection with a complete set of sc-discs on the other side of $H$ such that the point of intersection lies on a compressing or semi-compressing disc. The result of this isotopy is that the number of intersections of $H$ and $T$ is reduced by two.
\end{itemize}

\begin{figure}[ht]
\centering
\includegraphics[scale=.5]{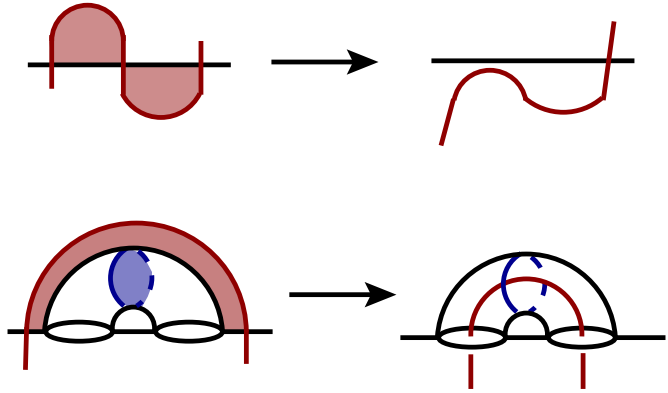}
\caption{The top row gives a schematic depiction of unperturbing. The bottom row gives a schematic depiction of removing a removable arc.}
\label{Fig:unperturb}
\end{figure}

\begin{definition}
Suppose that $\mc{H}$ is a multiple v.p.-bridge surface for $(M,T)$ and that some $(P, T_P) \cpt (M,T)\setminus \mc{H}$ is a trivial product compressionbody adjacent to $\mc{H}^-$ (rather than $\boundary M$). Then $\mc{H}' = \mc{H} \setminus \boundary P$ is obtained from $\mc{H}$ by a \defn{consolidation} (or by \defn{consolidating} $\mc{H}$.)
\end{definition} 

In \cite[Section 4]{PartI}, we show that the result of performing any of the generalized destabilizations, unperturbing, or removing a removable arc  is still a (oriented) multiple v.p.-bridge surface. We say that a multiple v.p.-bridge surface $\mc{H}$ is \defn{reduced} if it is impossible to perform a generalized destabilization or consolidation and if it is impossible to unperturb it or undo a removable arc. Let $\H(M,T)$ denote all reduced, oriented multiple v.p.-bridge surfaces, with two surfaces being equivalent if they are isotopic via an isotopy transverse to $T$.
 
There are two other ways of ``simplifying'' a multiple v.p.-bridge surface. In some ways, these play the most important role in the theory. They correspond to weak reduction of a Heegaard splitting (see \cite{CG, ST-thin, HS01}), but using sc-discs instead of compressing discs. For more detail and motivation, see \cite{PartI}.

\begin{definition}
Suppose that $\mc{H}$ is an oriented v.p.-bridge surface such that $H \cpt \mc{H}^+$ is sc-weakly reducible in $(M,T)\setminus \mc{H}^-$. Let $D_-$ and $D_+$ be disjoint sc-discs on opposite sides of $H$. Let $H_\pm$ be the result of compressing $H$ using $D_\pm$ and performing a small isotopy to the side of $H$ containing $D_\pm$. Let $F$ be the result of compressing $H$ using both $D_-$ and $D_+$. Let $\mc{J}^+ = (\mc{H}^+ \setminus H) \cup (H_- \cup H_+)$ and $\mc{J}^- = \mc{H}^- \cup F$. Then $\mc{J} = \mc{J}^+ \cup \mc{J}^-$ is obtained by \defn{untelescoping} $\mc{H}$. Figure \ref{Fig: untelescope} gives an example.
\end{definition}

\begin{center}
\begin{figure}[tbh]
\labellist \small\hair 2pt 
\pinlabel {$H$} [r] at 4 75
\pinlabel {$H+$} [r] at 360 150
\pinlabel {$H+$} [l] at 567 94
\pinlabel {$H-$} [r] at 325 48
\pinlabel {$H-$} [l] at 572 47
\pinlabel {$F$} [r] at 332 72
\pinlabel {$F$} [l] at 566 71
\pinlabel {$D_+$} [b] at 57 147
\pinlabel {$D_-$} [t] at 129 21
\endlabellist
\includegraphics[scale=0.5]{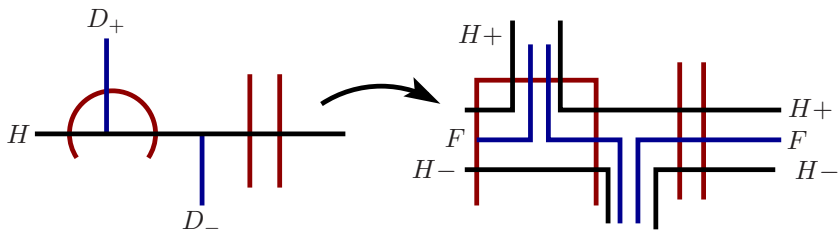}
\caption{Untelescoping $H$. The red curves are portions of $T$. The blue lines on the left are sc-discs for $H$. Note that if a semi-cut or cut disc is used then a ghost arc is created. This figure was reappropriated from \cite{PartI}.}
 \label{Fig:  untelescope}
\end{figure}
\end{center}

Corollary 5.9 of \cite{PartI} show that untelescoping an oriented multiple v.p.-bridge surface results in an oriented multiple v.p.-bridge surface. We need the following:

\begin{lemma}[Lemma 5.7 of \cite{PartI}]\label{lem:Sep weak reduction}
Suppose that  $H$ is a (oriented) multiple v.p.-bridge surface for $(M,T)$ with $(H_\down, T_\down)$ and $(H_\up, T_\up)$ the v.p.-compressionbodies of $(M,T)\setminus H$ on either side. Suppose that $D_\up$ and $D_\down$ are an sc-weak reducing pair. Let $H_- \subset H_\down$ and $H_+ \subset H_\up$ be the new thick surfaces created by untelescoping $H$. Let $F$ be the union of the new thin surfaces. Then the following are equivalent for a component $\Phi$ of $F$:
\begin{enumerate}
\item $\Phi$ is adjacent only to a remnant of $D_\up$ (or $D_\down$, respectively).
\item The disc $D_\up$ (or $D_\down$, respectively) is separating and $\Phi$ bounds a trivial product v.p.-compressionbody in $H_\up$ (or $H_\down$, respectively) with a component of $H_+$ (or $H_-$, respectively.)
\end{enumerate}
\end{lemma}

Suppose that $\mc{H} \in \H(M,T)$. An \defn{elementary thinning sequence} consists of the following operations, in order:
\begin{enumerate}
\item untelescoping a component of $\mc{H}^+$ to create an oriented multiple v.p.-bridge surface $\mc{H}_1$;
\item consolidating components of $\mc{H}_1 \setminus \mc{H}$ that cobound trivial product compressionbodies in $(M,T)\setminus \mc{H}_1$ to create $\mc{H}_2$
\item consolidating components of $\mc{H}_2^+ \setminus \mc{H}^+$ and components of $\mc{H}^- \subset \mc{H}_2^-$ that cobound trivial product compressionbodies in $(M,T)\setminus \mc{H}_2$ to create $\mc{H}' = \mc{H}_3$.
\end{enumerate}

See \cite[Figure 12]{PartI} for a schematic depiction of an elementary thinning sequence. It follows from \cite[Corollary 5.9]{PartI} that the result of applying an elementary thinning sequence to $\mc{H} \in \H(M,T)$ is an oriented multiple v.p.-bridge surface having the property that no consolidation is possible. It may, however, be possible to destabilize, unperturb, or undo a removable arc. 

\begin{definition}[Definition 6.15 of \cite{PartI}]
Suppose that $\mc{H} \in \H(M,T)$ is reduced and that $T$ is irreducible. An \defn{extended thinning move} applied to $\mc{H}$ consists of the following steps in the following order:
\begin{enumerate}
\item Perform an elementary thinning sequence
\item Destabilize, unperturb, and undo removable arcs until no generalized stabilizations, perturbations, or removable arcs remain
\item Perform as many consolidations as possible
\item Repeat (2) and (3) as much as possible.
\end{enumerate}
\end{definition}

In \cite{PartI} we define a certain complexity which decreases under each consolidation, elementary thinning sequence, destabilization, unperturbing, and undoing of a removable arc. This complexity ensures that steps (2), (3), and (4) are guaranteed to terminate and that there is no infinite sequence of extended thinning moves.

The result of applying an extended thinning move to $\mc{H} \in \H(M,T)$ is also an element of $\H(M,T)$. For $\mc{H}, \mc{K} \in \H(M,T)$, we say that $\mc{H}$ \defn{thins to} $\mc{K}$ and write $\mc{H} \more \mc{K}$ if there is a (possibly empty) sequence of extended thinning moves producing $\mc{K}$ from $\mc{H}$. If $\mc{H} \more \mc{K}$ implies that $\mc{H} = \mc{K}$ (equivalently, no extended thinning move can be applied to $\mc{H}$) then we say that $\mc{H}$ is \defn{locally thin}. We proved:

\begin{theorem}[Theorem 6.18 of \cite{PartI}]\label{partial order}
Suppose that $(M,T)$ is a 3-manifold graph pair satisfying the running assumption and with $T$ irreducible. Then $\more$ is a partial order on $\H(M,T)$ and for every $\mc{H} \in \H(M,T)$ there exists a locally thin $\mc{K} \in \H(M,T)$ such that $\mc{H} \more \mc{K}$.
\end{theorem}

Locally thin multiple v.p.-bridge surfaces have some particularly nice properties:

\begin{theorem}[Properties of locally thin surfaces]\label{Properties Locally Thin}
Suppose that $(M,T)$ is a (3-manifold, graph) pair with $T$ irreducible. Let $\mc{K} \in \H(M,T)$ be locally thin. Then the following hold:
\begin{enumerate}
\item Each component of $\mc{K}^+$ is sc-strongly irreducible in $(M,T)\setminus\mc{K}^-$. 
\item Every component of $\mc{K}^-$ is c-essential in $(M,T)$.
\item If $(M,T)$ is irreducible and if $\mc{K}$ contains a 2-sphere disjoint from $T$, then $T = \nil$ and $M = S^3$ or $M = B^3$.
\item Suppose that $P \subset (M,T)$ is an essential sphere such that $|P \cap T| \leq 3$. Then some $F \cpt \mc{K}^-$ is an essential sphere with $|F \cap T| \leq |P \cap T|$ and $F$ can be obtained from $P$ by a sequence of isotopies and compressions using sc-discs.
\end{enumerate}
\end{theorem}

See \cite[Theorem 7.6]{PartI} and \cite[Theorem 8.2]{PartI} for proofs.

\section{Net extent, width, and thinning sequences}
\subsection{Effects of thinning}

In this section we show that net extent and width do not increase under thinning of an oriented multiple v.p.-bridge surface. 

\begin{lemma}\label{lem:Thinning invariance}
Suppose that $\mc{H}$ is an oriented multiple v.p.-bridge surface for $(M,T)$ and that $\mc{K}$ is obtained by an elementary thinning sequence from $\mc{H}$. Then $\netchi(\mc{K}) = \netchi(\mc{H})$ and $\netextent(\mc{K}) = \netextent(\mc{H})$. Furthermore, if the following hold:
\begin{itemize}
\item $T$ is irreducible, and
\item either at least one of discs in the weak reducing pair has boundary which separates $\mc{H}^+$ or the union of the boundaries of the discs in the weak-reducing pair is non-separating on $\mc{H}^+$
\end{itemize}
then $w(\mc{K}) \leq w(\mc{H})$. 
\end{lemma}
\begin{proof}
The statement for net euler characteristic is similar to that of \cite[Lemma 2]{ScharlemannSchultens-Tunnel}; it is easily verified by examining the definition of elementary thinning sequence.  We take on the proof that net extent does not change and that width does not increase under an elementary thinning sequence. Observe that consolidation will never change net euler characteristic, net extent, or width.

Let $H \subset \mc{H}^+$ be the thick surface which is untelescoped using a weak reducing pair $\{D_-, D_+\}$. Let $i = |D_+ \cap T|$ and $j = |D_- \cap T|$ (so $i, j \in \{0,1\}$). Let $\mc{H}_1$ be the surface obtained by untelescoping. The surface $\mc{K}$ is obtained from $\mc{H}_1$ by consolidations so $\netextent(\mc{H}_1) = \netextent(\mc{K})$ and  $w(\mc{H}_1) = w(\mc{K})$. It suffices to show then that $\netextent(\mc{H}_1) = \netextent(\mc{H})$ and $w(\mc{H}_1) \leq w(\mc{H})$. 

Let $H_+$ be the union of the components $\mc{H}_1^+$ resulting from compressing $H$ using $D_+$ (there are at most two such components.)  Let $H_-$ be the union of the thick surfaces in $\mc{H}_1^+$ which result from compressing $H$ using $D_\down$.  Let $F$ be the union of the new thin surfaces (i.e. the components of $\mc{H}_1^- \setminus \mc{H}^-$.) We have
\[
\begin{array}{rcl}
\extent(H_+) &=& \extent(H) + i - 1 \\
\extent(H_-) &=& \extent(H) + j - 1 \\
\extent(F) &=& \extent(H) + i + j - 2.
\end{array}
\]
Consequently,
\[
\extent(H_+) + \extent(H_-) - \extent(F) = \extent(H),
\]
and so $\netextent(\mc{H}_1) = \netextent(\mc{H})$. Since extent is linear over components, we have $\netextent(\mc{H}) = \netextent(\mc{K})$.

We need to exert more care with width. Assume, therefore, the two additional hypotheses in the statement of the lemma. The second new hypothesis guarantees that $F = \mc{K}^- \setminus \mc{H}^-$ is connected (Lemma \ref{lem:Sep weak reduction}). Let $H'_+$, $H'_-$, and $F'$ be the components of $H_\pm \cap \mc{K}$ and $F \cap \mc{K}$ respectively. Let $H''_\pm$ be a component of $H_\pm$ which is consolidated with a component $F''_\pm$ of $F$.

Let $x = \extent(H)$. Let $x'_\pm$ and $x''_\pm$ be the extents of $H'_\pm$ and $H''_\pm$ respectively. Let $y$ be the extent of $F'$.  Note that the extents of the components of $F$ which are consolidated are exactly $x''_+$ and $x''_-$. Then
\[
\begin{array}{rcl}
x'_+  &=& x + i - 1 - x''_+ \\
x'_- &=& x + j -  1 - x''_- \\
y &=& x + i + j - 2 - x''_+ - x''_- \\
\end{array}
\]
Algebra then shows that
\[
(x'_+)^2 + (x'_-)^2  - y^2 =  x^2 - 2\big((j-1) - x''_-\big)\big((i-1) - x''_+\big).
\]

Thus, \[\frac{1}{2}\left(\width(\mc{K}) - \width(\mc{H})\right) = - 2\big((j-1) - x''_-\big)\big((i-1) - x''_+\big).\] This is non-positive, as desired, unless exactly one of $\big((j-1) - x''_-\big)$ or $\big((i-1) - x''_+\big)$ is positive and the other is negative. Without loss of generality, suppose
\[
x_-'' < j - 1 \in \{-1, 0\}.
\]
Since $T$ is irreducible, and since $H''_-$ is connected, by the definition of extent, $H''_-$ is a sphere disjoint from $T$. Hence, $x''_- = -1$ and $j = 0$. Thus,
\[
\frac{1}{2}\left(\width(\mc{K}) - \width(\mc{H})\right) = - 2\big(-1 + 1 )\big((i-1) - x''_+\big) \leq 0,
\]
as desired.
\end{proof}

\begin{remark}
The reason for the second additional assumption for the result on width in Lemma \ref{lem:Thinning invariance} is due to the fact that if (using the notation from the proof) $F'$ is disconnected, then the difference $\width(\mc{K}) - \width(\mc{H})$ is given by 
\[((x'_+)^2 + (x'_-)^2  - y_1^2 - y_2^2)- x^2\]
 where $y_1$ and $y_2$ are the extents of the components of $F'$ instead of $((x'_+)^2 + (x'_-)^2  - y^2)- x^2$. That distinction is enough to make the proof not go through in the case when $F'$ is disconnected.
\end{remark}

For convenience, if $\mc{H}$ is a reduced multiple v.p.-bridge surface for $(M,T)$ we define the \defn{width hypothesis} to be all three of the following assumptions:
\begin{enumerate}
\item[(W1)] $T$ is irreducible
\item[(W2)] Either each component of $\mc{H}^+$ has genus at most 2 or every closed surface in $M$ separates
\item[(W3)] Every sphere in $M$ separates
\end{enumerate}

Of course, we continue to employ the ``running assumption,'' without remarking on it.

\begin{corollary}\label{thm:Thinning invariance}
Suppose that $\mc{H}, \mc{K} \in \H(M,T)$ and that $\mc{H} \more \mc{K}$. Then $\netchi(\mc{K}) \leq \netchi(\mc{H})$ and $\netextent(\mc{K}) \leq \netextent(\mc{H})$. Additionally, if the width hypothesis holds for $\mc{H}$, then $\width(\mc{K}) \leq \width(\mc{H})$.
\end{corollary}
\begin{proof}
It is easily verified that consolidation, destabilization of all kinds, as well as unperturbing and eliminating a removable arc do not increase net euler characteristic, net extent, or width. By Lemma \ref{lem:Thinning invariance}, therefore, a thinning sequence does not increase net extent.

Assume, therefore, that the width hypothesis holds for $\mc{H}$. Observe that consolidation, elementary thinning sequences, destabilization of all kinds, as well as unperturbing and eliminating a removable arc do not change these properties. We will show that the width hypothesis implies that $\mc{H}$ satisfies the hypothesis in the second bullet point of the statement of Lemma \ref{lem:Thinning invariance}. It then follows that $\width(\mc{K}) \leq \width(\mc{H})$.

Suppose that $D_-$ and $D_+$ are a weak reducing pair for $H \cpt \mc{H}^+$. If one of $D_-$ or $D_+$ has boundary which separates $H$, then we are done. Assume, therefore, that both $\boundary D_-$ and $\boundary D_+$ are non-separating on $H$. This implies $H$ is not a sphere. By (W2), either $g(H) \leq 2$ or every closed surface in $M$ separates. 

Assume, first, that $g(H) \leq 2$. We already know that $g(H) \neq 0$. If $g(H) = 1$, then $H$ is a torus and $\boundary D_-$ and $\boundary D_+$ are parallel curves on $H$ (ignoring $T \cap H$.) Thus, $M$ contains a non-separating sphere, contradicting (W3). Hence $g(H) = 2$.

Since $\boundary D_-$ is non-separating on the genus 2 surface $H$, the surface $H' = H \setminus \boundary D_-$ is a genus one surface with two boundary components. If $\boundary D_+$ is non-separating on $H'$, then $\boundary D_- \cup \boundary D_+$ is non-separating on $H$, and we are done. Thus, we may assume that $\boundary D_+$ separates $H$. Together with components of $\boundary H'$, the curve $\boundary D_+$ must either bound a disc, a pair-of-pants, or an annulus in $H'$. Since $\boundary D_+$ is non-separating on $H$, we can rule out the first two possibilities. The third possibility implies that $\boundary D_+$ is parallel in $H$ to $\boundary D_-$ and so again $M$ contains a non-separating sphere, a contradicting (W3). Thus, the conclusion holds if $g(H) \leq 2$.

Suppose, therefore, that no closed surface in $M$ separates and that $\boundary D_-$ and $\boundary D_+$ are both non-separating on $H$, but $\boundary D_- \cup \boundary D_+$ is separating. Let $F_1$ and $F_2$ be the two components of $H \setminus (\boundary D_- \cup \boundary D_+)$. Since $D_-$ is non-separating in the v.p.-compressionbody $H_\down$ below $H$, there exists a properly-embedded arc $\psi_-$ in $H_-$ joining $F_1$ to $F_2$ which is disjoint from $D_-$. Likewise, there is a properly-embedded arc in the v.p.-compressionbody $H_\up$ above $H$ which joins $F_1$ and $F_2$ and is disjoint from $D_+$. Since $F_1$ and $F_2$ are each path-connected, without loss of generality, the endpoints of $\psi_-$ and $\psi_+$ coincide. Then $\psi_- \cup \psi_+$ is a loop intersecting each of $F_1$ and $F_2$ exactly once. Thus,  each of the components of the thin surface obtained by untelescoping $H$ using $D_-$ and $D_+$ are non-separating, a contradiction. Thus, the conclusion holds in this case also.
\end{proof}

\section{Minimality Results}

In this section we show that both net extent and width (at least under the width hypothesis) are non-negative and detect the unknot. The results and techniques of this section are often applicable more generally - for instance in studying links or graphs of small net extent. 

We begin by confirming that the net euler characteristic of a multiple v.p.-bridge surface provides an upper bound on the negative euler characteristic of its components. We use the fact that generalized Heegaard splittings (in the sense of Scharlemann-Thompson \cite{ST-thin}) can be amalgamated to create a Heegaard surface, a result due to Schultens \cite{Schultens}. We defer to Schultens' paper for a precise definition of amalgamation. (See also \cite{Lackenby}.)

\begin{lemma}\label{lem:Component bound}
Assume that $\mc{H}\in \H(M, T)$.  If $\netchi(\mc{H}) = x \in \Z$, then every component $S \cpt \mc{H}$ has $-\chi(S) \leq x$.
\end{lemma}
\begin{proof}
As $\netchi(\mc{H})$ is computed without taking $T$ into account, we may ignore $T$ for the purposes of this proof. Cap off all 2-sphere boundary components of $\boundary M$ with 3-balls, and consolidate parallel thick and thin surfaces in $\mc{H}$ as much as possible to obtain a multiple v.p.-bridge surface $\mc{J}$ for $M$. Observe that $\netchi(\mc{J}) = \netchi(\mc{H}) = x$.  Since we are ignoring $T$, we may amalgamate  $\mc{J}$ to a Heegaard surface $J$ for $M$. It is straightforward to verify that $-\chi(J) = x$. Each component of $\mc{J}$ is the component of a surface $J'$ obtained by a sequence of compressions of  $J$. Thus, $-\chi(J') \leq x$. Hence, any component  $S$ of $\mc{H}$ that is not consolidated away in the creation of $\mc{J}$ has $-\chi(S) \leq x$. Now reconstruct $\mc{H}$ from $\mc{J}$ by inserting in parallel thick and thin surfaces. Suppose that $F_1$ and $H_1$ are the first pair of parallel thick and thin surfaces inserted into $M \setminus \mc{J}$ (i.e. the last pair consolidated from $\mc{H}$). Let $H$ be the component of $\mc{J}^+$ which is adjacent to $F_1$ in $M \setminus (\mc{J} \cup H_1 \cup F_1)$. By the definition of v.p.-compressionbody, $-\chi(H) \geq -\chi(F_1)$, since there is a v.p.-compressionbody such that $H = \boundary_+ C$ and $F_1 \subset \boundary_- C$. Consequently, $-\chi(H_1) = -\chi(F_1) \leq -\chi(H) \leq x$. Proceeding inductively, we show that reinserting all pairs of parallel thin and thick surfaces we see that every component $S$ of $\mc{H}$ has $-\chi(S) \leq x$.
\end{proof}

\subsection{Compressionbodies}
In this subsection, we determine various inequalities for v.p.-compressionbodies. In future sections we will assemble these to study net extent and width.

For a v.p.-compressionbody $(C, T_C)$ with $T_C$ a 1--manifold, define
\[
\delta(C, T_C) = \extent(\boundary_+ C) - \extent(\boundary_- C)
\]
Note that $\delta(B^3, \nil) = - 1$ and if $(C, T_C)$ is any other trivial compressionbody or $(S^1 \times D^2, \nil)$ or $(S^1 \times D^2, \text{ core loop})$ then $\delta(C, T_C) = 0$.

\begin{lemma}\label{Lem: extent difference}
Suppose that $(C, T_C)$ is a v.p.-compressionbody other than $(B^3, \nil)$ or $(S^1 \times D^2, \nil)$. Assume $T_C$ is a 1--manifold not intersecting any spherical component of $\boundary_- C$ exactly once. Then $\delta(C, T_C) \geq 0$ and if $\delta(C, T_C) = 0$ then there is no compressing or semi-compressing disc for $\boundary_+ C$ in $(C, T_C)$.
\end{lemma}

\begin{proof}
Suppose that $(C, T_C)$ is other than $(B^3, \nil)$ or $(S^1 \times D^2, \nil)$. Let $\Delta$ be a complete set (possibly empty) of pairwise non-parallel sc-discs for $(C, T_C)$ such that reducing $(C, T_C)$ along $\Delta$ results in the union $(P, T_P)$ of trivial product compressionbodies. Let $a$ be the number of compressing and semi-compressing discs in $\Delta$ and let $b$ be the number of $(B^3, \nil)$ components of $(P, T_P)$. Since $(C, T_C)$ is not $(B^3, \nil)$ or $(S^1 \times D^2, \nil)$, each $(B^3, \nil)$ component of $(P, T_P)$ is adjacent to at least 3 remnants of compressing or semi-compressing discs. Hence, $2a \geq 3b$. 

Observe that
\[
\extent(\boundary_+ C) = \extent(\boundary_+ P) + a = a - b \geq a/3.
\]
Also we have $\extent(\boundary_- C) = \extent(\boundary_- P)$ so
\[
\delta(C, T_C) \geq a/3 \geq 0.
\]
Furthermore, if equality holds, then $a = 0$ and $\Delta$ does not contain a compressing or semi-compressing disc. Since this is true for every complete collection $\Delta$, $\boundary_+ C$ does not admit a compressing or semi-compressing disc in $(C, T_C)$ (Lemma \ref{Lem: Invariance}).
\end{proof}

The next definition will be useful for analyzing v.p.-compressionbodies.

\begin{definition}
Suppose that $(C, T_C)$ is a v.p.-compressionbody. The \defn{ghost arc graph} for $(C, T_C)$ is the graph $G$ whose vertices are the components of $\boundary_- C$ and whose edges are the ghost edges of $T_C$.
\end{definition}

\begin{corollary}\label{delta zero}
Suppose that $(C, T_C)$ is a v.p.-compressionbody. Assume $T_C$ is a 1--manifold not intersecting any spherical component of $\boundary_- C$ exactly once and that $\delta(C, T_C) = 0$. Then $(C, T_C)$ is one of the following 

\begin{enumerate}
\item $(B^3, \text{ arc})$ or 
\item $(S^1 \times D^2, \nil)$, or 
\item $(S^1 \times D^2, \text{ core loop})$ or 
\item a compressionbody such that every component of $T_C$ is a vertical arc or ghost arc and $g(\boundary_+ C) = g(\boundary_- C) + n - (|\boundary_- C| - 1)$, where $n$ is the number of ghost arcs in $T_C$. Furthermore, the ghost arc graph is connected.
\end{enumerate}
\end{corollary}
\begin{proof}
Assume that $(C, T_C)$ is not $(B^3, \text{ arc})$, $(S^1 \times D^2, \nil)$ or $(S^1 \times D^2, \text{core loop})$. Since $\delta(C, T_C) = 0$, $(C, T_C) \neq (B^3, \nil)$. We may apply Lemma \ref{Lem: extent difference}. A bridge arc in $T_C$ would imply the existence of a compressing or semi-compressing disc for $\boundary_+ C$. Similarly, if $T_C$ contained a closed loop, we would also have a compressing or semi-compressing disc as $(C, T_C) \neq (S^1 \times D^2, \text{ core loop})$. Thus, $T_C$ is the union of vertical arcs and ghost arcs.

If $T_C$ is the (possibly empty) union of vertical arcs, the assumption that $\delta(C, T_C) = 0$ and $(C, T_C) \neq (S^1 \times D^2, \nil)$ implies that $(C, T_C)$ is a trivial product compressionbody and the lemma follows. If $\boundary_- C$ is disconnected, there must be ghost arcs joining the components as otherwise $\boundary_+ C$ would have a compressing or semi-compressing disc as $\boundary_+ C$ can be obtained from $\boundary_- C$ by attaching 1-handles together with their cores. Thus $|\boundary_- C| - 1$ ghost arcs are needed to guarantee that $\boundary_+ C$ has no compressing or semi-compressing disc and each ghost arc beyond $|\boundary_- C| - 1$ increases $g(\boundary_+ C)$ by 1.
\end{proof}

Observe, that by Corollary \ref{delta zero}, if $(C,T_C)$ is a v.p.-compressionbody of Type (4), with $g(\boundary_+ C) = g(\boundary_- C)$, then the ghost arc graph is a (possibly empty) tree. If $T_C$ is non-empty and irreducible, and if the components of $\boundary_- C$ are spheres, then each leaf of the ghost arc graph must be incident to a vertical arc component of $T_C$.

We will piece the previous observations together with the following equations. Observe that if $\mc{H}$ is oriented, then since each component of $\mc{H}$ is adjacent to precisely two components of $(M,T)\setminus \mc{H}$:

\begin{equation}\label{net extent eqn}
2\netextent(\mc{H}) - \extent(\boundary M) - \sum_v (n_v - 2)/2 = \sum\limits_{(C, T_C)} \delta(C, T_C) 
\end{equation}
where the sum on the left is over all vertices $v$ in $T$ and $n_v$ is the valence of the vertex $v$ and the sum on the right is over all components $(C, T_C)$ of $(M,T)\setminus \mc{H}$ after drilling out the vertices of $T$.

Similar considerations show that (using the same notation):

\begin{equation}\label{width eqn}
\width(\mc{H}) - \sum\limits_{F \cpt \boundary M} \extent^2(F) - \sum\limits_v (n_v - 2)^2/4 = \sum\limits_{(C, T_C)} \extent^2(\boundary_+ C) - \sum\limits_{F \cpt\boundary_- C} \extent^2(F). 
\end{equation}

\begin{corollary}[Non-negativity]\label{Cor: non-negativity for surfaces}
Suppose that $(M,T)$ is irreducible  and satisfies the running assumption. Let $\mc{H} \in \H(M,T)$. Assume that no component of $(M,T)\setminus \mc{H}$ is $(B^3, \nil)$. Then 
\[ 
\netextent(\mc{H}) \geq \frac{1}{2}\big(\extent(\boundary M) -\chi(T) + |T \cap \boundary M|/2\big) \geq 0
\]
and
 \[
 \width(\mc{H}) \geq \netextent(\mc{H}).
\]
\end{corollary}

\begin{proof} 
By equation \eqref{net extent eqn}, we have:
\[
2\netextent(\mc{H}) =  \extent(\boundary M) + \sum_v (n_v - 2)/2 + \sum\limits_{(C, T_C)} \delta(C, T_C) 
\]

Notice that $\delta(S^1 \times D^2, \nil) = 0$ and recall that no component of $(M,T) \setminus \mc{H}$ is $(B^3, \nil)$.  Since $T$ is irreducible, by Lemma \ref{Lem: extent difference}, $\delta(C, T_C) \geq 0$ for all $(C, T_C) \cpt (M,T) \setminus \mc{H}$. Suppose that $T$ has $V$ interior vertices, $E$ edges, and $n = T \cap \boundary M$. We have
\[
2E = n + \sum_v n_v.
\]
Thus, recalling that $-\chi(T) = E - (V + n)$:
\[\begin{array}{rcl}
2\netextent(\mc{H}) &\geq& \extent(\boundary M) + E - n/2 - V \\
&=& \extent(\boundary M) - \chi(T) + n/2 \\
\end{array}
\]

To see that this is non-negative, rewrite the previous equation as
\[\begin{array}{rcl}
\extent(\boundary M) + E - n/2 - V &=& \frac{-\chi(\boundary M)}{2} + \frac{n}{2} + E - \frac{n}{2} - V \\
&=& \frac{-\chi(\boundary M)}{2} + E - V. 
\end{array}
\]
Let $(\mathring{M}, \mathring{T})$ be the result of drilling out the interior vertices of $M$. This does not change the number of edges of the graph, but converts the interior vertices into spherical boundary components. Thus,
\begin{equation}\label{only link}
\frac{-\chi(\boundary M)}{2} + E - V = \frac{-\chi(\boundary \mathring{M}) + |\boundary \mathring{M} \cap \mathring{T}|}{2}.
\end{equation}
By the running assumption, each spherical component of $\boundary \mathring{M}$ intersects $\mathring{T}$ at least three times. Thus, 
\[
\extent(\boundary M) - \chi(T) + n/2 \geq 0.
\]
Furthermore, using Equation \eqref{only link} and the running assumption, we have equality only if $T$ has no internal vertices (i.e. is empty or is a link) and every component of $\boundary M$ is a torus.

We now consider width. Recall that for each $F \cpt \mc{H}$, $\extent(F) \geq 0$ since $T$ is irreducible and since no component of $\mc{H}$ is a sphere disjoint from $T$. Thus, for any component $(C, T_C)$ of $(M,T)\setminus \mc{H}$, we have
\[\begin{array}{rcl}
\extent^2(\boundary_+ C) - \sum\limits_{F \cpt \boundary_- C} \extent^2(F) &\geq & \extent^2(\boundary_+ C) - \extent^2(\boundary_- C) \\
&=& (\extent(\boundary_+ C) + \extent(\boundary_+ C))\delta(C, T_C) \\
\end{array}.
\]

If $(\extent(\boundary_+ C) + \extent(\boundary_+ C)) = 0$, then $\extent(F) = 0$ for each component $F$ of $\boundary C$. This implies that $\delta(C, T_C) = 0$. Thus,
\[
\extent^2(\boundary_+ C) - \sum_{F \cpt \boundary_- C} \extent^2(F) \geq \delta(C, T_C).
\]

By Equation \eqref{width eqn} and by Equation \eqref{net extent eqn}, we have:
\[ \begin{array}{rl}
\width(\mc{H}) &=\\
\sum\limits_{F \cpt \boundary M} \extent^2(F) + \sum\limits_v (n_v - 2)^2/4 + \sum\limits_{(C, T_C)} \left(\extent^2(\boundary_+ C) - \sum\limits_{F \cpt \boundary_- C} \extent^2(F)\right) & \geq \\
\sum\limits_{F \cpt\boundary M} \extent(F) + \sum\limits_v (n_v - 2)/4 + \sum\limits_{(C, T_C)} \delta(C, T_C) &\geq \\
\frac{1}{2}\left(\extent(\boundary M) + \sum_v (n_v - 2)/2 + \sum\limits_{(C, T_C)} \delta(C, T_C) \right) &=\\
\netextent(\mc{H}).&
\end{array}
\]
Hence,
\[
\width(\mc{H}) \geq \netextent(\mc{H}),
\]
as desired.
\end{proof}

\begin{remark}\label{Rem: zero delta}
Observe from the proof of Corollary \ref{Cor: non-negativity for surfaces}, that if $\netextent(\mc{H}) = (\extent(\boundary M) - \chi(T) + n/2)/2$  and if $(C, T_C)$ is obtained from a component of $(M,T)\setminus \mc{H}$ by drilling out vertices, then $\delta(C, T_C) = 0$. We will make use of this in the next corollary. Furthermore, it follows from the proof that if $\netextent(\mc{H}) = 0$, then $T$ is either empty or a link and $\boundary M$ is the (possibly empty) union of tori.
\end{remark}

\begin{corollary}\label{Equality}
Assume that $(M,T)$ is irreducible and satisfies the running assumption, with $T$ a 1--manifold, and is other than $(S^3, \nil)$. Suppose that $\mc{H}\in \H(M,T)$ is locally thin with either $\netextent(\mc{H})$ or $\width(\mc{H})/2$ equal to 
\[
\frac{1}{2}\left(\extent(\boundary M) -\chi(T) + |T \cap \boundary M|/2\right)
\] 
Then for every component $(C, T_C)$ of $(M,T)\setminus \mc{H}$ we have $\delta(C, T_C) = 0$. In particular, each $(C, T_C)$ is one of:
\begin{enumerate}
\item $(B^3, \text{ arc})$
\item $(S^1 \times D^2, \nil)$
\item $(S^1 \times D^2, \text{ core loop})$
\item a compressionbody such that every component of $T_C$ is a vertical arc or ghost arc and $g(\boundary_+ C) = g(\boundary_- C) + n - (|\boundary_- C| - 1)$, where $n$ is the number of ghost arcs in $T_C$. Furthermore, the ghost arc graph is connected.
\end{enumerate}
\end{corollary}
\begin{proof}
Observe that $(M,T)$ is not $(B^3, \nil)$ since it satisfies the running assumption. By Theorem \ref{Properties Locally Thin}, since $(M,T)$ is irreducible and is not $(S^3, \nil)$ or $(B^3, \nil)$, no component of $(M,T)\setminus \mc{H}$ is a $(B^3, \nil)$. We thus satisfy the hypothesis of Corollary \ref{Cor: non-negativity for surfaces}. If we have equality for $\width(\mc{H})/2$, then by that corollary we also have equality for net extent. By Remark \ref{Rem: zero delta}, $\delta(C, T_C) = 0$ for each $(C, T_C) \cpt (M,T)\setminus\mc{H}$. By Lemma \ref{delta zero}, the result follows.
\end{proof}

We minimize net extent and width to show that they are non-negative half-integer valued invariants of (3-manifold, graph) pairs satisfying the running assumption.

\begin{corollary}\label{Cor: well-defined}
Suppose that $(M,T)$ is an irreducible (3-manifold, graph) pair satisfying the running assumption, other than $(S^3, \nil)$. Let $x \geq 2g(M) - 2$. Then 
\[
\netextent_x(M,T) \geq \frac{\extent(\boundary M) -\chi(T)}{2} + |\boundary M \cap T|/4 \geq 0.
\]
Furthermore, if  every sphere in $M$ separates and if either $x \leq 2$ or every closed surface in $M$ separates, then
\[
\width_x(M,T) \geq \netextent_x(M,T).
\]
\end{corollary}

\begin{proof}
Let $H$ be a minimal genus Heegaard surface for $M$ so that $-\chi(H)= 2g(M) - 2$. Isotope $H$ to be disjoint from the vertices of $T$. Drill out the vertices of $T$ to obtain $(\mathring{M}, \mathring{T})$ and observe that $H$ is still a Heegaard surface for $\mathring{M}$. It is a standard result (cf. \cite[Lemma 2.1]{HS-coreloop}) that $\mathring{T}$ can be isotoped to intersect the compressionbodies on either side of $H$ in bridge arcs and vertical arcs. Filling the vertices of $T$ back in, the surface $H$ is a multiple v.p.-bridge surface for $(M,T)$. Performing generalized destabilizations, unperturbations, consolidations, and undoing removable edges shows that $\H(M,T) \neq \nil$ and that there is an element with net euler characteristic at most $2g(M) - 2$. (In fact, $\netchi(\mc{H}) = 2g(M) - 2$.)  Let $\mc{H} \in \H(M,T)$. By Theorem \ref{partial order}, there exists $\mc{K} \in \H(M,T)$ such that $\mc{H} \more \mc{K}$ and $\mc{K}$ is locally thin.  By Lemma \ref{thm:Thinning invariance}, $\netchi(\mc{K}) \leq \netchi(\mc{H})$. By Theorem \ref{Properties Locally Thin}, no component of $\mc{H}$ is a 2-sphere disjoint from $T$. Thus, no component of $(M,T)\setminus \mc{H}$ is $(B^3, \nil)$. By Corollary \ref{Cor: non-negativity for surfaces}, we have
\[
\netextent(\mc{H}) \geq \frac{1}{2}\big(\extent(\boundary M) -\chi(T) + n/2\big) \geq 0.
\]
Hence, 
\[
\netextent_x(M,T) \geq \frac{1}{2}\big(\extent(\boundary M) -\chi(T) + n/2\big) \geq 0,
\]
for all $x \geq 2g(M) - 2$. 

If  every sphere in $M$ separates and if either $x \leq 2$ or every closed surface in $M$ separates then the width hypothesis holds for $\mc{H}$. Thus, it holds also for $\mc{K}$. The result then follows as before.
\end{proof}

\subsection{Detecting the unknot}

In this subsection we show that net extent  and width detect the unknot. For our purposes, a \defn{Hopf link} in a lens space or $S^3$ is the union of the cores of the solid tori on either side of a genus 1 Heegaard surface.

\begin{theorem}[Detecting the unknot]\label{Net Extent Detects unknot}
Suppose that $(M,T)$ is an irreducible (3-manifold, graph) pair satisfying the running assumption such that $M$ is connected, every sphere in $M$ separates, and $T \neq\nil$. Also, assume that $(M,T)$ does not have a (lens space, core loop), (lens space, Hopf link), ($S^3$, Hopf link), or $(S^1 \times D^2, \text{core loop})$ connect summand. 

If $\netextent_x(M,T) = 0$ for some $x \geq 2g(M) - 2$, then $(M,T) = (S^3, \text{ unknot})$. Furthermore, if $x \leq 2$ or if every closed surface in $M$ separates, then the same result holds if $\width_x(M,T) = 0$.
\end{theorem}

\begin{proof}
We will show that the theorem holds for all $x$ with $2g(M) - 2 \leq x < \infty$. By definition, it will then also hold when $x = \infty$.  By Theorem \ref{Cor: well-defined}, if $x \leq 2$ or if $M$ contains no non-separating closed surface, then $\width_x(M,T) \geq \netextent_x(M,T)$. Consequently, we may assume that $\netextent_x(M,T) = 0$ for some $x \geq 2g(M) - 2$.

Let $\mc{H}\in \H(M, T)$ be such that $\netchi(\mc{H}) \leq x$ and $\netextent(\mc{H}) = \netextent_x(M, T) $. By Corollary  \ref{partial order} and Corollary \ref{thm:Thinning invariance}, we may also assume that $\mc{H}$ is locally thin. In particular, by Theorem \ref{Properties Locally Thin}, no component of $\mc{H}$ is a sphere disjoint from $T$ and no component of $(M,T) \setminus \mc{H}$ is a trivial product compressionbody adjacent to a component of $\mc{H}^-$. Furthermore, by Remark \ref{Rem: zero delta}, $T$ is a link and $\boundary M$ is the (possibly empty) union of tori. 

By Remark \ref{Rem: zero delta}, for all components $(C, T_C)$ of $(M,T)\setminus \mc{H}$ we have $\delta(C, T_C) = 0$. By Corollary \ref{delta zero}, if $(C,T_C)$ is a component of $(M,T)\setminus\mc{H}$, then $(C, T_C)$ is one of:
\begin{enumerate}
\item $(B^3, \text{ arc})$,
\item (solid torus, $\nil$),
\item (solid torus, core loop),
\item v.p.-compressionbody such that every component of $T_C$ is a vertical arc or ghost arc and there is no (semi-)compressing disc for $\boundary_+ C$ in the complement of $T$.
\end{enumerate}

\textbf{Case 1:} $T$ is disjoint from $\mc{H}$.

Then, $\mc{H}$ is a multiple v.p.-bridge surface for the exterior of $T$. Amalgamate $\mc{H}$ to a Heegaard surface $H$ for the exterior of $T$. It is easily verified that 
\[
-\chi(H) = \netchi(\mc{H} ) = 2\netextent(\mc{H}) = 0.
\]
Thus, $H$ is a torus is a torus disjoint from $T$. Since every sphere in $M$ separates, $M\neq S^1 \times S^2$. Thus, $(M,T)$ is one of:
\begin{itemize}
\item $(S^3, \text{ unknot })$
\item $(S^3, \text{ Hopf link })$
\item $(\text{ solid torus}, \text{ core loop })$
\item $(\text{ lens space}, \text{ core loop })$
\item $(\text{ lens space}, \text{ Hopf link })$ \qed(Case 1)
\end{itemize}

\textbf{Claim 1:} $\boundary M = \nil$

Suppose that $\boundary M \neq \nil$ and let $(C, T_C) \cpt (M,T)\setminus \mc{H}$ be a v.p.-compressionbody adjacent to a component of $\boundary M$. Since $T \cap \boundary M = \nil$ and there is no compressing disc for $\boundary_+ C$ in $(C,T_C)$, the v.p-compressionbody $(C,T_C)$ must be a product compressionbody with $T_C = \nil$.

Let $(D, T_D)$ be the v.p.-compressionbody, distinct from $(C, T_C)$, with $\boundary_+ D = \boundary_+ C$. If $\boundary_- D = \nil$, then $(D, T_D)$ is either $(S^1 \times D^2, \nil)$ or $(S^1 \times D^2, \text{ core loop})$. In the former case, $T = \nil$, a contradiction. In the latter case, $(M,T)$ is (solid torus, core loop), contradicting our assumption that $(M,T)$ has no (solid torus, core loop) summands. Thus, $\boundary_- D \neq \nil$.

Suppose $\boundary_- D$ contains a torus. By Corollary \ref{Equality}, $\boundary_- D$ is equal to that  torus and $T_D = \nil$. Since $T \neq \nil$, $\boundary_- D \subset \mc{H}^-$. Hence, $(D, T_D)$ is a trivial product compressionbody adjacent to $\mc{H}^-$. This contradicts the local thinness of $\mc{H}$. Consequently, $\boundary_- D$ is the non-empty union of spheres.

Let $\Gamma$ be the ghost arc graph. Since every  component of $\boundary_- D$ intersects $T$ at least twice and since $\boundary_+ D \cap T = \nil$, the ghost arc graph is a cycle. This implies that $(M,T)$ has a (solid torus, core loop) summand, a contradiction. We conclude that $\boundary M = \nil$.

\textbf{Claim 2:} No v.p.-compressionbody $(C, T_C)\cpt (M,T)\setminus \mc{H}$ is a (solid torus, $\nil$) or (solid torus, core loop).

This is similar to the proof of Claim 1. Suppose some $(C, T_C)$ is (solid torus, $\nil$) or (solid torus, core loop).  Let $(D, T_D) \neq (C, T_C)$ be the other v.p.-compressionbody adjacent to $\boundary_+ C$.  If $(D, T_D)$ is (solid torus, $\nil$) or  (solid torus, core loop) then $\mc{H} = \boundary_+ D = \boundary_+ C$ and we are in Case 1.  Thus, $\boundary_- D \neq \nil$. As in Case 1, $\boundary_- D$ is the non-empty  union of spheres, $T_D \neq \nil$, and $(D, T_D)$ is a v.p.-compressionbody of Type (4).

There is a cut-disc $E$ for $\boundary_+ D$ in $(D, T_D)$. Since $(C, T_C)$ is either a (solid torus, $\nil$) or a (solid torus, core loop) there is also a compressing disc or cut disc $E'$ for $\boundary_+ D = \boundary_+ C$ in $(C, T_C)$.  Isotope $E$ and $E'$ to minimize the number of intersections between their boundaries. If their boundaries are disjoint, then $M$ contains a non-separating sphere, a contradiction. If their boundaries intersect exactly once, then $E'$ must be a cut disc, as otherwise $\mc{H}$ would be meridionally stabilized. 

Let $P$ be the sphere which results from compressing $\boundary_+ C = \boundary_+ D$ using the cut disc $E$. Observe that $|P \cap T| = 2$ and $P$ bounds a submanifold $W$ of $M$ containing torus $\boundary_+ C = \boundary_+ D$, which is a genus 1 Heegaard surface for $W$. 

If $|\boundary E \cap \boundary E'| = 1$, then $W$ is a 3-ball. Otherwise, $W$ is a punctured lens space. The Heegaard surface $\boundary_+ C = \boundary_+ D$ in $W$ is disjoint from $T$. Indeed, after capping off $P \subset \boundary W$ with a $(B^3, \text{ arc})$, we obtain either $(S^3, \text{ Hopf link})$ or (lens space, Hopf link), or (lens space, core loop). This contradicts our initial assumption on $(M,T)$. Consequently, no component of $(M,T)\setminus \mc{H}$ is (solid torus, $\nil$) or (solid torus, core loop). 

We can now conclude the proof.

By Claims (1) and (2), each component of $(M,T) \setminus \mc{H}$ is either a trivial ball compressionbody or a v.p.-compressionbody of Type (4) above.  Recall that all spheres in $M$ separate. If $\mc{H}^-$ contains a sphere, choose such a sphere $P \cpt \mc{H}^-$, so that $P$ bounds a submanifold $W \subset M$ such that there are no spherical components of $\mc{H}^-$ in the interior of $W$. Since $M$ is closed, $\boundary W = P$.   If $\mc{H}^-$ does not contain any spheres, let $W = M$.

Consider a v.p.-compressionbody $(C, T_C) \cpt (W, T \cap W)\setminus \mc{H}$. We claim that $(C, T_C)$ is not a trivial ball compressionbody. If it is, let $(D, T_D)$ be the other v.p.-compressionbody such that $\boundary_+ D = \boundary_+ C$. Observe that $D \subset W$, since $C$ is a 3--ball contained in $W$ and disjoint from $P$. 

If $(D, T_D)$ is a trivial ball-compressionbody, then $(M,T) = (C, T_C) \cup (D, T_D) = (S^3, \text{ unknot})$, so suppose it is of Type (4). In particular, $\boundary_- D \neq \nil$. Since $\boundary_+ D$ is a sphere, each component of $\boundary_- D$ is a sphere. Consider the ghost arc graph $\Gamma$. By Corollary \ref{delta zero}, $\Gamma$ is a tree with the number of edges equal to $|\boundary_- D| - 1$. If there are no edges (i.e. $|\boundary_- D| = 1$), then $(D, T_D)$ would be a product, a contradiction. Hence, $|\boundary_- D| \geq 2$. But this contradicts our choice of $P$ to be innermost. Thus, $(C, T_C)$ is a v.p.-compressionbody of Type (4). In particular, $\boundary_- C \neq \nil$.

Without loss of generality, we may assume that the transverse orientation on $P$ points into $W$. Consider a thick surface $H \cpt \mc{H}^+ \cap W$. Let $(C_1, T_1) \cpt (M,T)\setminus \mc{H}$ be the v.p.-compressionbody such that the tranverse orientation on $H$ points into $C_1$. This implies that $P \not \subset \boundary_- C_1$. Since $\boundary_- C_1 \neq \nil$ and $\boundary_- C_1 \not \subset \boundary M$, there is another v.p.-compressionbody $(C_2, T_2)$, distinct from $(C_1, T_1)$, such that $\boundary_- C_1 \cap \boundary_- C_2 \neq \nil$. Thus, there is a flow line from $H= \boundary_+ C_1$ to the thick surface $\boundary_+ C_2$. Repeatedly applying this same argument, we can construct a flow line beginning at $P$ and intersecting $\mc{H}^+$ $n$ times, for any $n \in \N$. Since $\mc{H}^+$ has only finitely many components, we conclude that there is a closed flow line, contradicting the fact that $\mc{H}$ is oriented. 
 \end{proof}

\begin{remark}
The minimal bridge sphere for the unknot is a v.p.-bridge surface for the unknot, showing that $\netextent_{-2}(S^3, \text{ unknot }) = 0$. Similarly, if $M$ has a genus one Heegaard splitting (i.e. is $S^3$, $S^1 \times S^2$, a lens space, $T^2 \times I$, or solid torus) and if $L$ is a one or two component link which is the union of cores for the solid tori in the Heegaard splitting, then $\netextent_0(M,L) = 0$. However, as we will see, net extent and width are additive under connected sum and so taking the connected sum with (lens space, core loop) or (solid torus, core loop) will not change net extent or width. Thus, if $(M,T)$ is any irreducible 3--manifold pair satisfying the running assumption and with the property that every sphere in $M$ separates and $T\neq \nil$, and if $\netextent_x(M,T) = 0$ for $x\geq 0$, then $(M,T)$ is either  $(S^3, \text{ unknot})$, ($S^3$, Hopf link), (lens space, Hopf link), (lens space, core loop), (solid torus, core loop), or is a connected sum with each summand being one of those pairs.

The hypothesis that all spheres separate cannot be dropped. To see this, let $M = S^1 \times S^2$ and consider the closure of a 2-stranded braid $T$ in $M$. We can construct an oriented multiple v.p.-bridge surface $\mc{H}$ for $(M,T)$ as in Figure \ref{fig:Netext0ce}. This oriented multiple v.p.-bridge surface has $\netextent(\mc{H}) = 0$.
\end{remark}

\begin{figure}[ht]
\centering
\includegraphics[scale=0.5, angle=90]{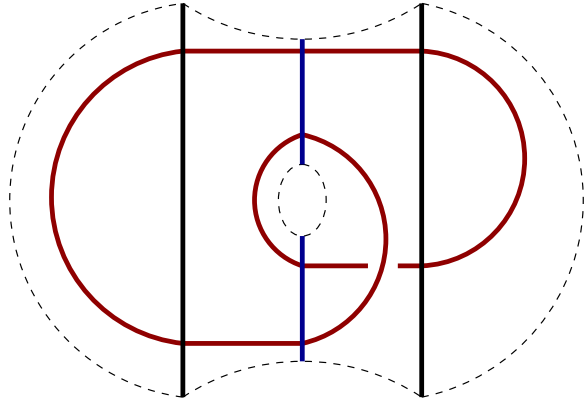}
\caption{An oriented multiple v.p.-bridge surface $\mc{H}$ for $(S^1 \times S^2, T)$ where $T$ is the closure of a  2-stranded braid. All the vertical lines are spheres, with each of the blue lines representing a non-separating sphere. The black lines are thick surfaces and the blue lines are thin surfaces. Each v.p.-compressionbody is either a trivial ball compressionbody or is a twice-punctured 3--ball containing a single vertical arc and a single ghost arc. The orientation on $\mc{H}$ isn't shown.}
\label{fig:Netext0ce}
\end{figure}

\section{Additivity of Net Extent and Width}
In this section we prove that net extent and width satisfy an additivity property with respect to connected sum and trivalent vertex sum. In fact, apart from some hypotheses on $M$ and $T$, the only properties of net extent and width that we use are that they are order preserving with respect to extended thinning sequences and that they depend on euler characteristic and the number of intersections with $T$. For convenience, therefore, and with a view to the fact that there are other invariants which have similar properties we prove our additivity theorem in a rather abstract setting. Theorem \ref{thm: Super-additivity} shows that super-additivity holds; Theorem \ref{subadditive} shows that sub-additivity holds; and Theorem \ref{Main Theorem} puts those together to show the additivity result for net extent and width.

We begin by relating the thin levels of a locally thin multiple v.p.-bridge surfaces to a prime decomposition.

\begin{proposition}\label{Thin summing spheres}
Assume that $(M,T)$ is non-trivial and that $T$ is irreducible. Suppose that $\mc{H} \in \H(M,T)$ is locally thin. Then there exists a subset $\mc{P} \subset \mc{H}^-$ such that $\mc{P}$ is the union of decomposing spheres giving a prime decomposition of $(M,T)$.
\end{proposition}
\begin{proof}
If $(M,T)$ is prime, then it is its own prime decomposition and setting $\mc{P} = \nil$, we are done. Assume that $(M,T)$ contains at least one essential twice or thrice punctured sphere.

Let $\mc{Q}$ be the union of all the twice and thrice-punctured spheres in $\mc{H}^-$. Let $(M_i, T_i)$ be a component of $(M,T) \setminus \mc{Q}$. Let $(\wihat{M}_i, \wihat{T}_i)$ be the result of capping off twice punctured spheres (corresponding to copies of components of $\mc{Q}$) in $\boundary M_i$ with a trivial $(B^3, \text{ arc})$ and capping off thrice punctured spheres  (corresponding to copies of components of $\mc{Q}$) of $F$ in $\boundary M_i$ with a 3-ball containing a boundary parallel tree with a single internal vertex. 

We claim that $(\wihat{M}_i, \wihat{T}_i)$ contains no essential twice or thrice-punctured sphere. Suppose, to the contrary, that such a sphere $F$ exists. Since each component of $\mc{Q}$ is essential in $(M,T)$, the surface $F$ is also essential in $(M_i, T_i)$. It is easy to see that the intersection $\mc{H}_i$ of $\mc{H}$ with the interior of $M_i$ is still a locally thin, linear, multiple v.p.-bridge surface for $(M_i, T_i)$.  Thus, by Theorem \ref{Properties Locally Thin}, there exists an essential twice or thrice-punctured sphere $P \cpt \mc{H}^-_i$. However, $\mc{H}^-_i \subset \mc{H}^-$ does not contain any essential twice or thrice-punctured spheres by the definition of $\mc{Q}$ and $\mc{H}_i$. 

Suppose that $(M_i, T_i)$ has the property that $(\wihat{M}_i, \wihat{T}_i)$ is $(S^3, \text{ unknot })$. Since each component of $\mc{Q}$ is essential in $(M,T)$, $\boundary M_i$ has multiple components, each a sphere intersecting $T$ twice. Let $P$ be one such component, and let $(M_j, T_j)$ be the component of $(M,T)\setminus \mc{Q}$ adjacent to $P$ and not equal to $(M_i, T_i)$. Let $(M', T') = (M_i, T_i) \cup (M_j, T_j)$ and let $(\wihat{M}', \wihat{T}')$ be the result of capping off components of $\boundary M'$ corresponding to components of $\mc{Q} \setminus P$. Then $(\wihat{M}', \wihat{T}')$ is the connected sum of $(\wihat{M}_j, \wihat{T}_j)$ with $(S^3, \text{ unknot })$. It is thus homeomorphic to $(\wihat{M}_j, \wihat{T}_j)$ and so does not contain an essential twice or thrice-punctured sphere. Continuing on in this vein, we may remove some number of components from $\mc{Q}$ to obtain $\mc{P}'$ such that if $(\wihat{M}', \wihat{T}')$ is obtained by capping off components of $\boundary M'$ corresponding to components of $\mc{P}'$, where $(M', T') \cpt (M,T)\setminus \mc{P}'$ then $(\wihat{M}', \wihat{T}')$ neither contains an essential twice or thrice-punctured sphere nor is $(S^3, \text{ unknot })$. 

Suppose now that $(M_i, T_i)$ has the property that $(\wihat{M}_i, \wihat{T}_i)$ is $(S^3, \wihat{T}_i)$ with $\wihat{T}_i$ a trivial $\theta$-graph and with some component $P$ of $\boundary M_i$ a thrice-punctured sphere. Let $(M_j, T_j) \cpt (M,T)\setminus \mc{P}'$ be on the other side of $P$ from $(M_i, T_i)$. Let $(M', T') = (M_i, T_i) \cup (M_j, T_j)$ and let $(\wihat{M}', \wihat{T}')$ be the result of capping off components of $\boundary M'$ corresponding to components of $\mc{P}' \setminus P$. Then $(\wihat{M}', \wihat{T}')$ is the trivalent vertex sum of $(\wihat{M}_j, \wihat{T}_j)$ with $(S^3, G)$, where $G$ is a trivial $\theta$-graph.  It is thus homeomorphic to $(\wihat{M}_j, \wihat{T}_j)$ and so does not contain an essential twice or thrice-punctured sphere. Continuing on in this vein, we can remove some number of components from $\mc{P}'$ to arrive at $\mc{P}$, the union of some components of $\mc{P}'$, which give a prime decomposition of $(M,T)$.
\end{proof}

Let $\M$ be a non-empty set whose elements are irreducible (3-manifold, graph) pairs $(M,T)$ satisfying the running assumption with $M$ connected such that  every sphere in $M$ separates. Suppose that $\M$ has the property that if $(M,T)  \in \M$ and $M = (M_1, T_1) \# (M_2, T_2)$ or $M = (M_1, T_1)\#_3 (M_2, T_2)$ then both $(M_1, T_1)$ and $(M_2, T_2)$ are also elements of $\M$. Let $\S$ denote the set of closed surfaces $S \subset (M,T)$ for some $(M,T) \in \M$ and let $\S_0 \subset \S$ be the subset of connected surfaces. Let $X = (\Z \times \Z) \cap ([-2, \infty)\times [0, \infty))$ and let $r \co \S_0 \to \R$ be any function which factors through the function $\S_0 \to X$ defined by $S \mapsto (-\chi(S), |T \cap S|)$. (That is, $r$ depends only on $-\chi(S)$ and $|S \cap T|$.) Extend $r$ to a function $r\co \S \to \R$ linearly, that is if $S_1, S_2 \in \S_0$ are disjoint then $r(S_1 \cup S_2) = r(S_1) + r(S_2)$. 

\begin{example}\label{r examples 1}
If $S \subset \S$ and $k \in \N$, then the function $r(S) = \sum_{i} \extent^k(S_i)$ where the sum is over all the components $S_i$ of $S$ is such a function.
\end{example}

For $(M,T) \in \M$, recall that $\H(M,T)$ is the set of reduced, linear, multiple v.p.-bridge surfaces for $(M,T)$. Let $\H = \bigcup\limits_{(M,T) \in \M} \H(M,T)$. Define $\net r \co \H \to \R$ by
\[
\net r(\mc{H}) = r(\mc{H}^+) - r(\mc{H}^-).
\]

If $x \geq 2g(M) - 2$ and $(M,T) \in \M$, we say that $x$ is \defn{realizable} for $(M,T)$. If $x$ is realizable for  $(M,T)$, let $\net r_x(M,T) = \min_{\mc{H}} \net r(\mc{H})$ where the minimum is over all $\mc{H} \in \H(M,T)$ such that $\netchi(\mc{H}) \leq x$. We say that $r$ is \defn{order-preserving} on $\H$ if whenever $\mc{H}, \mc{K} \in \H$ and $\mc{H} \more \mc{K}$, then $\net r(\mc{H}) \geq \net r(\mc{K})$. Let $r_2$ be the value of $r$ on a sphere twice punctured by $T$ and $r_3$ be the value of $r$ on a sphere thrice-punctured by $T$.

\begin{example}\label{r examples 2}
If we choose $\M$ to be the set of all irreducible $(M,T)$  satisfying the running assumption with $M$ connected and every sphere separating. Let $r\co \S_0 \to \R$ to be $r(S) = \extent(S)$, then $\net r = \netextent$ and $\net r_x(M,T) = \netextent(M,T)$. In this case, $r_2 = 0$ and $r_3 = 1/2$. Likewise, if we also insist that either $g(M) \leq 2$ or $M$ contains no closed non-separating surface and define $r \co \S_0 \to \R$ to be $r(S) = 2\extent^2(S)$, then $\net r = \width$ and $\net r_x(M,T) = \width_x(M,T)$. In this case also, $r_2 = 0$ and $r_3 = 1/2$.
\end{example}

Return to the general situation, where $\M$ is any non-empty set whose elements are irreducible (3-manifold, graph) pairs $(M,T)$ satisfying the running hypothesis and where $M$ is connected, and every sphere in $M$ separates. Assume also that $\M$ is closed under taking factors of connected sum and trivalent vertex sum. 

\begin{theorem}[Super-additivity]\label{thm: Super-additivity}
Suppose that $(M,T) \in \M$ is non-trivial and that $x$ is realizable for $(M,T)$ and that $r$ is order-preserving on $\H$. Then there is a prime factorization of $(M,T)$ into $(\wihat{M}_1, \wihat{T}_1), \hdots, (\wihat{M}_n, \wihat{T}_n)$ and so that there exist integers $x_1, \hdots, x_n$ summing to at most $x - 2(n-1)$ so that $x_i$ is realizable for $(\wihat{M}_i, \wihat{T}_i)$ and
\[
\net r_x(M,T) \geq  -p_2r_2 - p_3 r_3+ \sum_{i = 1}^n \net r_{x_i}(\wihat{M}_i, \wihat{T}_i).
\]
where $p_2$ is the number of connected sums and $p_3$ is the number of trivalent vertex sums in the decomposition.
\end{theorem}

\begin{proof}
Let $(M,T) \in \M$. Since $x$ is realizable, there exists $\mc{H} \in \H(M,T)$ such that $\netchi(\mc{H}) \leq x$ and  $\net r_x(M,T) = \net r(\mc{H})$. Let $\mc{K}\in \H(M,T)$ be locally thin so that $\mc{H} \more \mc{K}$. By Theorem \ref{partial order}, such a $\mc{K}$ exists. Since $r$ is order-preserving, $\net r(\mc{K}) \leq \net r(\mc{H})$. By Lemma \ref{thm:Thinning invariance}, \[\netchi (\mc{K}) \leq \netchi(\mc{H}) \leq x.\] Thus, by our choice of $\mc{H}$, $\net r(\mc{K}) = \net r(\mc{H})$.

By Proposition \ref{Thin summing spheres}, there is a subset $\mc{P} \subset \mc{K}^-$ which is the union of components such that $\mc{P}$ are the summing spheres giving a prime decomposition of $(M,T)$. Split $(M,T)$ open along $\mc{P}$ to obtain components $(M_1, T_1), \hdots, (M_n, T_n) \cpt (M,T)\setminus \mc{P}$. Let $(\wihat{M}_i, \wihat{T}_i)$ be the result of capping off the components of $\mc{P}$ in $\boundary M_i$ with trivial ball compressionbodies, so that $(\wihat{M}_1, \wihat{T}_1), \hdots, (\wihat{M}_n, \wihat{T}_n)$ give a prime decomposition of $(M,T)$. Let $p_2$ be the number of twice-punctured spheres in $\mc{P}$. Let $p_3$ be the number of thrice-punctured spheres in $\mc{P}$. Observe that $p_2 + p_3 = n - 1$. Let $x_i = \netchi(\mc{K}_i)$. Note that $x_i$  is realizable for $(\wihat{M}_i, \wihat{T}_i)$. Furthermore,
\[
\sum x_i = \netchi(\mc{K}) - 2(n-1) \leq  x - 2(n-1),
\]
since spheres have Euler characteristic equal to 2. Consequently, we have $x_i$ summing to at most $x - 2(n-1)$ and a prime decomposition $(\wihat{M}_i, \wihat{T}_i)$ of $(M,T)$ so that $\net r_{x_i} (\wihat{M}_i, \wihat{T}_i) \leq \net r (\mc{K}_i)$. Notice that:
\[
\net r_x(M,T) \geq \net r(\mc{K}) = \sum\limits_{H \cpt \mc{H}^+}  r(H) - \sum\limits_{F \cpt \mc{H}^-\setminus \mc{P}} r(S) - p_2 r_2 - p_3 r_3.
\]
Splitting up the sums according to which $(M_i, T_i)$ contains the surface  shows that
\[
\net r_x(M,T) \geq \sum_{i = 1}^n \net r(\mc{K}_i) - p_2 r_2 - p_3 r_3 \geq \sum_{i=1}^n \net r_{x_i}(\wihat{M}_i, \wihat{T}_i) - p_2 r_2 - p_3 r_3.
\]
\end{proof}

\begin{theorem}[Sub-additivity]\label{subadditive}
Let $(M,T) \in \M$ is non-trivial and that $\net r$ does not increase under consolidation, generalized destabilization, unperturbing, or undoing a removable arc. Suppose that $(M,T)$ is the connected sum and trivalent vertex sum of non-trivial pairs $(\wihat{M}_1, \wihat{T}_1), \hdots, (\wihat{M}_n, \wihat{T}_n)$, that $x$ is realizable for $(M,T)$ and so that there are integers $x_1, \hdots, x_n$ such that each $x_i$ is realizable for $(\wihat{M}_i, \wihat{T}_i)$ and $x_1 + \cdots + x_n \leq x - 2(n-1)$. Then
\[
\net r_x(M,T) \leq \sum_{i=1}^n \net r_{x_i}(\wihat{M}_i, \wihat{T}_i) - p_2r_2 - p_3 r_3,
\] 
\end{theorem}

\begin{figure}[ht]
\labellist \small\hair 2pt 
\pinlabel {$G_1$} [l] at 135 332
\pinlabel {$G$} [b] at 445 471 
\pinlabel {The result of inserting $G_1$ into $G$} [r] at 325 114
\endlabellist 
\centering
\includegraphics[scale=.5]{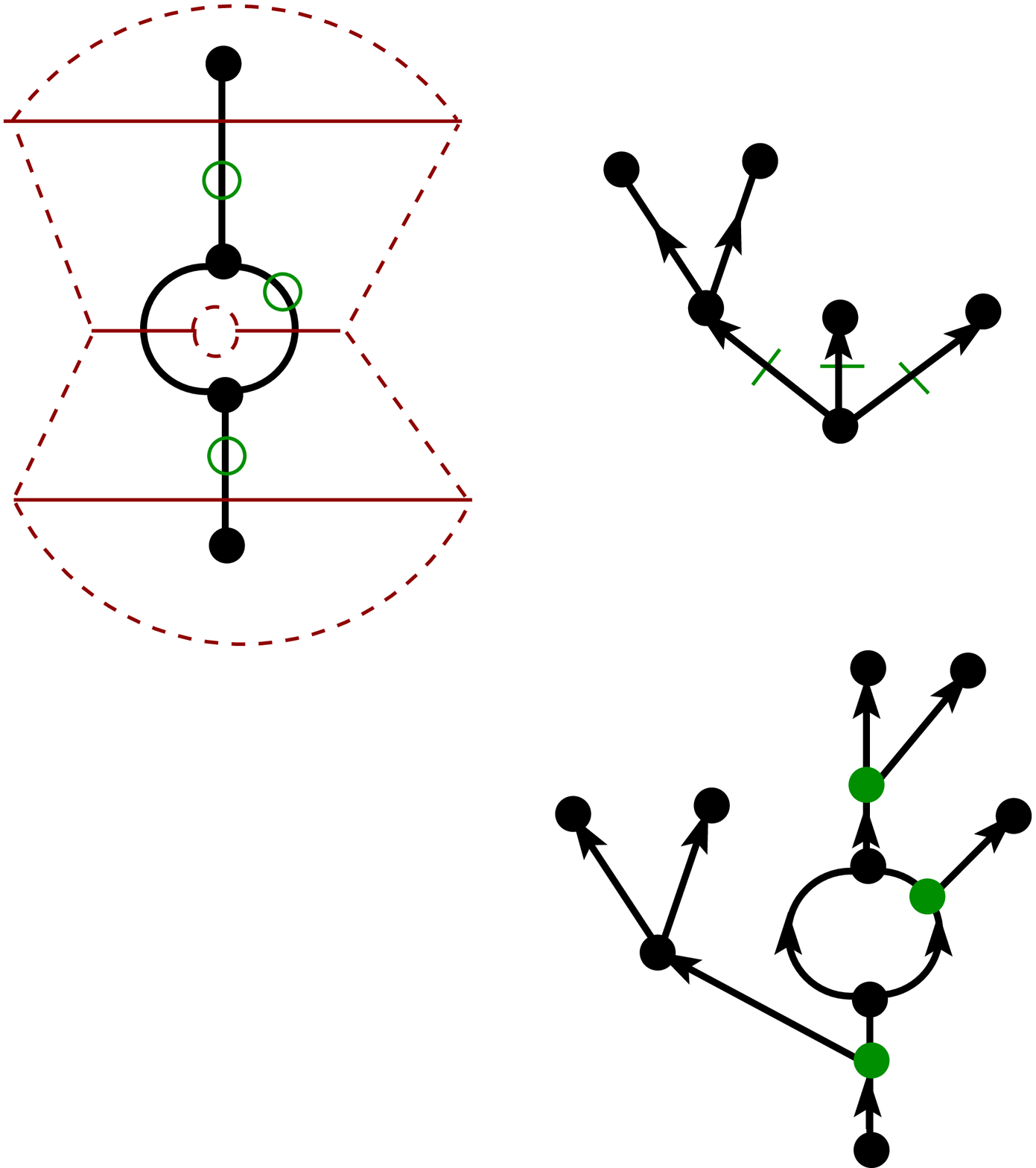}
\caption{The first step of turning $\mc{H}$ into an oriented v.p.-bridge surface. We insert the graph $G_1$ into the tree $G$ at the root, ensuring that the orientations of the edges are consistent. The green circles, lines, and dots indicate the points $p_1$ and the spheres $P_1$.}
\label{fig: trees}
\end{figure}

\begin{proof}
Recall that, by the definition of $\M$, each $(\wihat{M}_i, \wihat{T}_i) \in \M$. For each $i \in \{1, \hdots, n\}$, let $p_i \subset T_i$ be the union of points where $(\wihat{M}_i, \wihat{T}_i)$ is summed to one of the other (3-manifold, graph)-pairs. For each $i$, $|p_i| \geq 1$.  The graph in $M$ dual to the summing spheres is a finite tree. All finite trees have one more vertex than edge. Hence, we have $\sum_{i=1}^n |p_i|  = 2(n-1)$. For each $i$, let $\mc{H}_i \in \H(\wihat{M}_i, \wihat{T}_i)$ be such that $\netchi(\mc{H}_i) \leq x_i$ and $\net r (\mc{H}_i) = \net r_{x_i}(\mc{H}_i)$.

By general position, we may assume that $\mc{H}_i \cap p_i = \nil$. Notice that we can create another oriented multiple v.p.-bridge surface $\mc{H}'_i$ for $(\wihat{M}_i, \wihat{T}_i)$ by reversing all the transverse orientations. We call this \defn{turning $\mc{H}_i$ upside down}. Turning $\mc{H}_i$ upside down does not change $\net r(\mc{H}_i)$ or $\netchi(\mc{H}_i)$.

Let $(M_i, T_i)$ be the result of removing a small open regular neighborhood of $p_i$ from $(\wihat{M}_i, \wihat{T}_i)$. Let $P_i$ be the union of the components of $\boundary M_i$ corresponding to the points $p_i$. Each component of $P_i$ is a twice or thrice-punctured sphere. We may view each $(M_i, T_i)$ as embedded in $(M,T)$ with $P_i \subset M$ the union of separating essential twice and thrice-punctured spheres.

Let $\mc{H} = \bigcup_i \mc{H}_i \cup P_i$. Clearly $\mc{H}$ is a v.p.-bridge surface. We will show that, perhaps after turning some of the $\mc{H}_i$ upside down, we can define a transverse orientation so that $\mc{H}$ is an oriented multiple v.p.-bridge surface for $(M,T)$. 

By the definition of connected sum and trivalent-vertex sum, the graph $G$ in $M$ dual to the summing spheres is a tree. Each vertex of $G$ is  some $(M_i, T_i)$ and we associate the midpoint of each edge of $G$ with some component $P$ of some $P_i$.   Let $(M_1, T_1)$ be the root of $G$ and put a partial order $\leq$ on the vertices of $G$ so that $(M_1, T_1)$ is the least element of the partial order and if a vertex $c$ separates vertices $a$ and $b$ then $a < c < b$. Orient the edges of $G$ so that if vertices $v$ and $w$ are the endpoints of an edge pointing from $v$ to $w$ then $v < w$. 

Let $G_i$ be the graph in $(M_i, T_i)$ dual to $\mc{H}_i$. The transverse orientation on $\mc{H}_i$ induces an orientation on the edges of $G_i$. Suppose that $P \subset P_i$ is a component. For each $i \in \{1, \hdots, n\}$, replace the vertex $(M_i, T_i)$ in $G$ with the graph $G_i$; we obtain the graph $G''$ dual to $\mc{H}$. See Figure \ref{fig: trees}. Since $G$ was a tree, after turning $\mc{H}_i$ upside down, if necessary, $G''$ becomes an oriented graph, inducing transverse orientations on the components of $P_i$ for each $i$ and $\mc{H}$ becomes a multiple v.p.-bridge surface with coherent transverse orientations. Since $G$ is a tree and each $\mc{H}_i$ is an oriented multiple v.p.-bridge surface, there is no closed flow line in $M$. Thus, $\mc{H}$ is an oriented multiple v.p.-bridge surface for $(M,T)$.

We now do the necessary calculations to obtain our bound.

Observe that $x = \netchi(\mc{H}) = \sum_{i=1}^n \netchi(\mc{H}_i) - 2(n-1)$. Furthermore, since $r$ depends only on negative euler characteristic and the number of intersections with $T$,
\[\begin{array}{rcl}
\net r (\mc{H}) &=& \sum_{i=1}^n \net r(\mc{H}_i) - p_2 r_2 - p_3 r_3 \\
&=& \sum_{i=1}^n \net r(\wihat{M}_i,\wihat{T}_i) - p_2r_2 - p_3 r_3.
\end{array}
\] 
By assumption, consolidation, generalized destabilization, unperturbing, and undoing a removable arc to do not increase $\net r$. Thus, we may perform these operations on $\mc{H}$ as necessary to ensure it is reduced. Then, $\net r_x(M,T) \leq  \net r(\mc{H})$, and the result follows.
\end{proof}

\begin{corollary}[Additivity]\label{General additivity}
Assume that $r$ is order-preserving and that $x$ is realizable for some non-trivial $(M,T) \in \M$. Then there is a prime factorization of $(M,T)$ into $(\wihat{M}_1, \wihat{T}_1), \hdots, (\wihat{M}_n, \wihat{T}_n)$ such that there exist integers $x_1, \hdots, x_n$, summing to at most $x - 2(n-1)$, so that $x_i$ is realizable for $(\wihat{M}_i, \wihat{T}_i)$, and
\[
\net r_x(M,T) =  -p_2r_2 - p_3 r_3+ \sum_{i = 1}^n \net r_{x_i}(\wihat{M}_i, \wihat{T}_i).
\]
where $p_2$ is the number of twice punctured spheres and $p_3$ is the number of thrice punctured spheres in the decomposition. 
\end{corollary}
\begin{proof}
The corollary follows immediately from the super-additivity and sub-additivity theorems (Theorems \ref{thm: Super-additivity} and \ref{subadditive}.)
\end{proof}

It is now easy to verify that net extent and width are additive. Let $\M$ be the set of irreducible (3-manifold, graph) pairs $(M,T)$ satisfying the running assumption such that $T \neq \nil$ and every $S^2 \subset M$ separates. Let $\M_2 \subset \M$ be the subset with elements $(M,T)$ such that $g(M) \leq 2$. Let $\M_s \subset \M$ be the subset with elements $(M,T)$ where every closed surface in $M$ separates.

\begin{theorem}[Net Extent and Width are Additive]\label{Main Theorem}
Let $(M,T) \in \M$ be non-trivial,  let $g$ be the Heegaard genus of $M$, and let $x$ be any integer with $x \geq 2g - 2$. Then there is a prime factorization of $(M,T)$ into $(\wihat{M}_1, \wihat{T}_1), \hdots, (\wihat{M}_n, \wihat{T}_n)$ so that there exist integers $x_1, \hdots, x_n$, summing to at most $x - 2(n-1)$,  with $x_i$ is realizable for $(\wihat{M}_i, \wihat{T}_i)$ and
\[
\netextent_x(M,T) =  -p_3/2+ \sum_{i = 1}^n \netextent_{x_i}(\wihat{M}_i, \wihat{T}_i).
\]
where $p_3$ is the number of thrice punctured spheres in the decomposition. Furthermore, if $(M,T) \in \M_s$ or if $(M,T) \in \M_2$ and $x \leq 2$, then also
\[
\width_x(M,T) =  -p_3/2+ \sum_{i = 1}^n \width_{x_i}(\wihat{M}_i, \wihat{T}_i).
\]
\end{theorem}

\begin{proof}
As in Examples \ref{r examples 1} and \ref{r examples 2}, $r = \extent$ and $\net r = \netextent$ satisfy the requirement that for $S \in \S_0$, $r(S)$ depends only on the euler characteristic and number of punctures of $S$.  By Corollary \ref{thm:Thinning invariance}, extent is order-preserving. By Corollary \ref{General additivity}, we have the result for net extent. If $(M,T) \in \M_s$ or if $(M,T) \in \M_0$ and $x \leq 2$, then a similar argument shows that $\width_x$ is additive. 
 \end{proof}

\section{Comparison with Gabai thin position}\label{comparison}

The width for knots in $S^3$ defined by Gabai \cite{G3} and our definition of $\width_{-2}$ applied to pairs $(S^3,K)$ are very similar. Both definitions have thick surfaces $\mc{H}^+$ and thin surfaces $\mc{H}^-$  that are spheres and have a height function. Both widths can be calculated via similar formulae. Gabai's width is given \cite[Lemma 6.2]{SS-width} by the formula:
\[
\frac{1}{2}\left(\sum_{H \cpt \mc{H}^+} |H \cap K|^2 - \sum_{F \cpt \mc{H}^-} |F \cap K|^2\right)
\]
and our width is given by
\[\begin{array}{lc}
2\left(\sum_{H \cpt \mc{H}^+} \frac{(|H \cap K| - 2)^2}{4} - \sum_{F \cpt \mc{H}^-} \frac{(|F \cap K| - 2)^2}{4}\right) &= \\
\frac{1}{2}\left(\sum_{H \cpt \mc{H}^+} (|H \cap K| - 2)^2 - \sum_{F \cpt\mc{H}^-} (|F \cap K| - 2)^2\right).
\end{array}
\]
Finally, both definitions of width are related to a definition of thin position. Indeed, we can say that $\mc{H}$ is in Gabai thin position if $\mc{H}$ minimizes Gabai's width for a knot $K$. Similarly, with our definitions there is always a $\mc{H}$ which is both locally thin and minimizes $\width_{-2}$. 

And yet Gabai thin position is not necessarily additive under connected sum \cite{BT} but our width is (Theorem \ref{Main Theorem}). The essential difference between the two definitions of width is that in Gabai thin position all the components of $\mc{H} = \mc{H}^+ \cup \mc{H}^-$ are concentric, while in our definition the components of $\mc{H}$ need not be concentric.

We now briefly examine Blair and Tomova's counterexample to width additivity for Gabai thin position in light of our definition. Figure \ref{fig:Gabaithin} shows a knot $K$ (in fact a family of knots) and the connect sum of $K\#\text{trefoil}$. Note that the projections of $K$ and $K\#\text{trefoil}$ depicted in the figure have the same Gabai width while the trefoil has a Gabai width of 8. The crux of showing that this is indeed a counterexample to additivity of Gabai width is to show that the embedding of $K$ depicted in Figure \ref{fig:Gabaithin} is actually in Gabai thin position. The thin and thick surfaces in the figure are a v.p.-multiple bridge surface $\mc{H}$. As $\width_{-2}$ is additive, it must be the case that $\mc{H}$ is not a minimum width multiple v.p. bridge surface for $K$. Note that $$w(\mc{H})=2(4^2+4^2+4^2-1^2 -1^2)=92$$

\begin{center}
\begin{figure}[tbh] 
\centering
\labellist \small\hair 2pt 
\pinlabel {$K$} [ul] at 280 40
\pinlabel {$K\#\text{trefoil}$} [ul] at 650 40
\endlabellist 
\includegraphics[scale=.45]{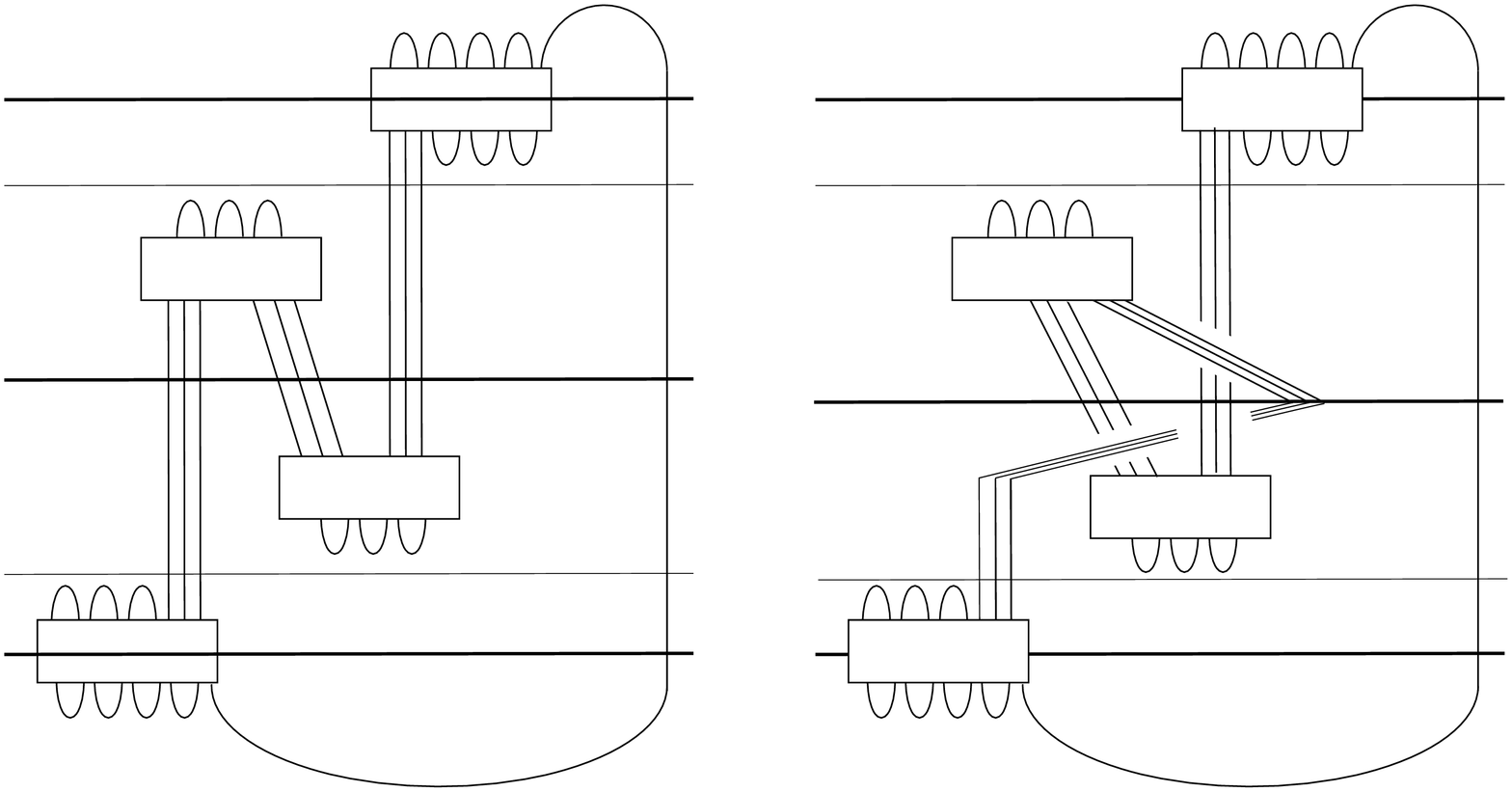}
\caption{The rectangles represent particular braids, which are irrelevant for our purposes. Thick and thin surfaces are represented with thick and thin lines respectively.}
\label{fig:Gabaithin}
\end{figure}
\end{center}

Another projection $K'$ of the knot $K$ is depicted on both the left and right of Figure \ref{fig:thinning}. That $K$ and $K'$ are isotopic was noted by Scharlemann and Thompson in \cite{SchTh-width}. To show that the multiple bridge surface $\mc{H}$ on the left of Figure \ref{fig:thinning} is not locally thin, we point out (again on the left of Figure \ref{fig:thinning}) a weak-reducing pair of discs for each thick surface. Applying two elementary thinning sequences using the indicated discs produces (after an isotopy) the multiple v.p. bridge surface $\mc{H}'$ depicted on the right of Figure \ref{fig:thinning}. Using our formula for width
\[
\width(\mc{H}')=2(2^2+4^2+4^2+2^2-1^2-1^2-1^2)=74.\]
This demonstrates that $\mc{H}$ is indeed not a minimum width multiple v.p. bridge surface for $K$, although it does minimize Gabai width (for particular choices of braids).

\begin{center}
\begin{figure}[tbh] 
\labellist \small\hair 2pt 
\pinlabel {$\mc{H}$} [ul] at 130 40
\pinlabel {$\mc{H}'$} [ul] at 550 40
\endlabellist 
\includegraphics[scale=.5]{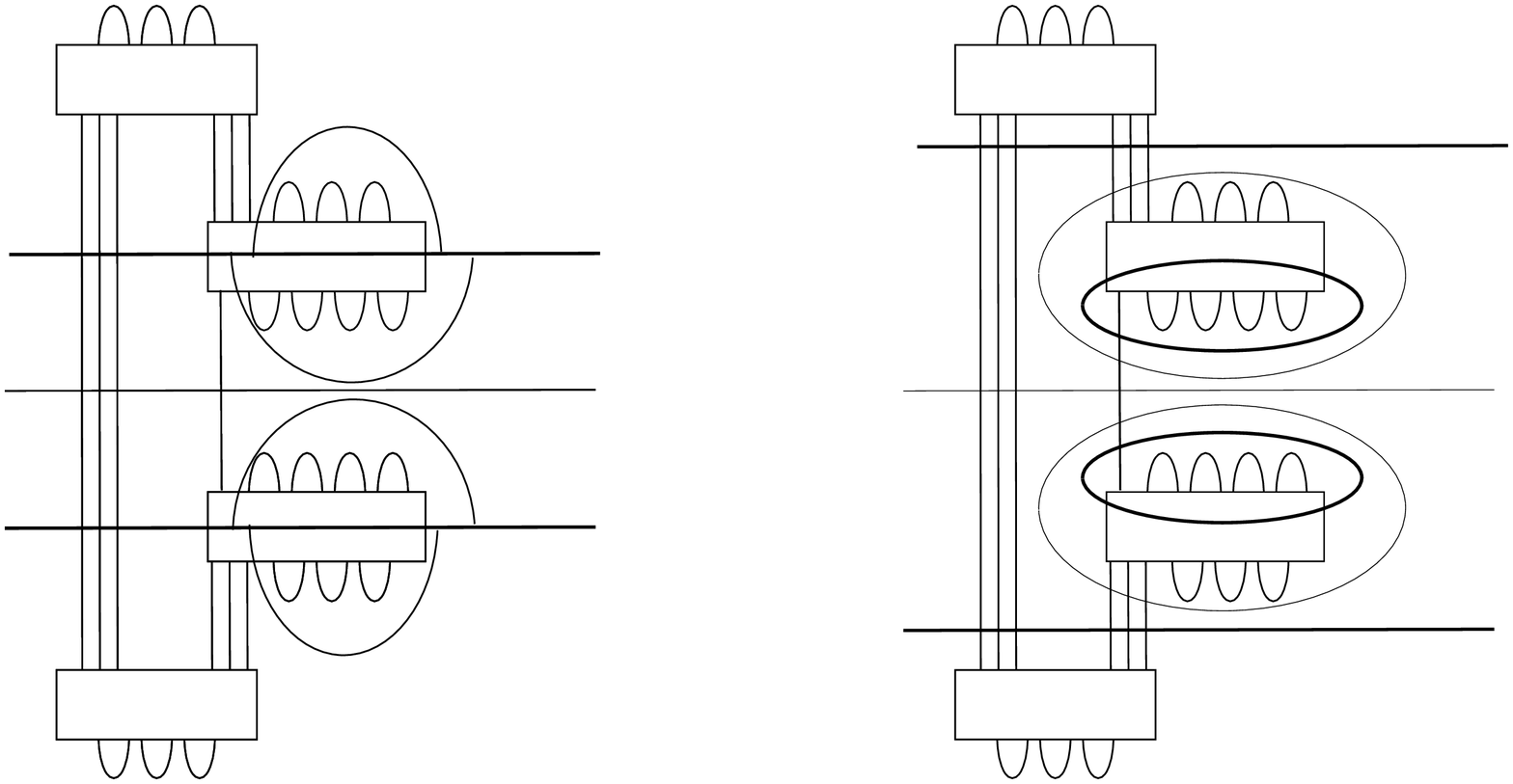}
\caption{Thick and thin surfaces are represented with thick and thin lines respectively.}
\label{fig:thinning}
\end{figure}
\end{center}

\section{On some classical invariants}\label{classical}

As an easy example of how net extent can be used to study classical invariants, we reprove classical theorems of Schubert \cite{Schubert} and Norwood \cite{Norwood}. The Schubert theorem is also a  consequence of the fact that the double branched cover over a 2-bridge knot is a lens space. We include a proof, however, as an example of how to use our techniques.

\begin{theorem}[Schubert]\label{thm:schubert}
If $K \subset S^3$ is a knot which is 2-bridge with respect to a sphere, then $K$ is prime.
\end{theorem}

\begin{proof}
Suppose that $K$ is a composite 2-bridge knot. Since $K$ is 2-bridge, \[\extent_{-2}(S^3,K) \leq 1.\] Since extent is always integral for knots in $S^3$, by Theorem \ref{Net Extent Detects unknot}, $\extent_{-2}(S^3,K) = 1$.  Since $K$ is composite, by Theorem \ref{Main Theorem}, it has a prime factorization $K = K_1 \# \hdots \# K_n$ such that
\[
1 = \netextent_{-2}(S^3,K) = \netextent_{-2}(S^3,K_1) + \hdots + \netextent_{-2}(S^3,K_n).
\]
Since each $K_i$ is non-trivial, by Theorem \ref{Net Extent Detects unknot} it follows that: 
\[
1 \geq n.
\]
Thus, $n = 1$ and so $K$ is prime.
\end{proof}

Recall that, after assigning a transverse orientation, a Heegaard surface for the exterior of a knot $K$ in a 3-manifold $M$ is a oriented multiple v.p.-bridge surface for $(M,K)$. Thus, $\netextent_\infty(M,K) \leq t(K)$. 

\begin{theorem}[Norwood]\label{thm:norwood}
If $K \subset S^3$ has $t(K) = 1$, then $K$ is prime.
\end{theorem}
\begin{proof}
Let $K$ be a knot with $t(K) = 1$. Since the unknot has tunnel number 0, $K$ is not the unknot. By Theorem \ref{Main Theorem}, $K$ has a prime factorization
\[K = K_1 \# \hdots \# K_n\]
such that 
\[
1 \geq \netextent_{x}(K) = \netextent_{x}(K_1) + \hdots + \netextent_{x}(K_n).
\]
Since each $K_i$ is non-trivial, by Theorem \ref{Net Extent Detects unknot} it follows that: 
\[
1 \geq n.
\]
Thus, $n = 1$ and so $K$ is prime.
\end{proof}

Scharlemann and Schultens \cite{ScharlemannSchultens-Tunnel} generalized Norwood's theorem to show that if a knot $K \subset S^3$ has at least $n$ prime factors, then $K$ has tunnel number at least $n$. (Another proof has been given by Weidmann \cite{Weidmann}.) Morimoto \cite{Morimoto} showed that the tunnel number of meridionally small knots does not go down under connected sum. Here is a common generalization of both the Scharlemann-Schultens and Morimoto result.

\begin{theorem}\label{MSS generalization}
For each $i \in \{1, \hdots, n\}$ let $K_i$ be a knot in a closed, orientable 3-manifold $M_i$ such that every sphere in $M_i$ separates and each $(M_i, K_i)$ is prime and irreducible. Assume that there is an integer $j \leq n$ is such that $(M_i, K_i)$ is meridionally small if and only if $i \leq j$. Then, letting $(M,K) = (M_1, K_1) \# \cdots \# (M_n, K_n)$, we have:
\[
(n - j) + t(K_1) + \hdots + t(K_j) \leq t(K) \leq (n-1) + \sum t(K_i)
\]
\end{theorem}

\begin{proof}
Recall that Heegaard genus is additive under connected sum of 3--manifolds. Consequently, we may assume that the exterior of $K_i$ in $M_i$ (for each $i$) is irreducible. Since every sphere in every $M_i$ separates, each pair $(M_i, K_i)$ is irreducible. Furthermore, by \cite[Theorem 4.1]{Miyazaki}, the prime factorization
\[
(M_1, K_1) \# \cdots \# (M_n, K_n)
\]
is unique up to re-ordering.

Let $t$ be the tunnel number of $K$ and let $H$ be a minimal genus Heegaard surface for the exterior of $K$. Let $x = \infty$. The surface $H$ is also a v.p.-bridge surface for $(S^3, K)$ and so
\begin{equation}\label{net extent and tunnel number}
\netextent_{x}(M,K) \leq \extent(H) = t(K).
\end{equation}

By Theorem \ref{Main Theorem} and our assumption on the uniqueness of prime factorization, there exist $(M_1, K_1)$, $\hdots$, $(M_n, K_n)$ such that there are integers $x_1, \hdots, x_n$ so that all of the following hold:
\begin{enumerate}
\item $x_i$ is realizable for $(M_i, K_i)$ for all $i$.
\item $x_1 + \hdots + x_n \leq x - 2(n-1)$
\item 
\begin{equation*}\label{additivity application}
\netextent_{-2}(M, K) = \sum_{i=1}^n \netextent_{x_i} (M_i, K_i).
\end{equation*}
\end{enumerate}
Suppose $(M_i, K_i)$ is meridionally small and let $\mc{H}_i$ be a multiple v.p.-bridge surface for $(M_i, K_i)$ such that $\netchi(\mc{H}_i) \leq x_i$ and $\netextent(\mc{H}_i) = \netextent_{x_i}(M_i, K_i)$. If there were a component of $\mc{H}^-$ which intersected $K_i$, by Theorem \ref{Properties Locally Thin}, we would contradict local thinness. Thus, $K \cap \mc{H}^-_i = \nil$ and there is at most one component $H \subset \mc{H}_i^+$ which intersects $H$. By performing $|H \cap K_i|/2$ meridional stabilizations on $H$, we may create a surface $H'$ such that $\mc{H}_i' = (\mc{H}_i \setminus H) \cup H'$ is a multiple v.p.-bridge surface for $(M_i, K_i)$ disjoint from $K_i$. Observe that $\netextent(\mc{H}_i') = \netextent(\mc{H}_i)$. Since $\mc{H}'_i$ is disjoint from $K_i$ we may amalgamate \cite{Schultens-amalgamate} $\mc{H}'_i$ to a Heegaard surface $J_i$ for the exterior of $K_i$. It is easy to verify that 
\begin{equation}\label{string}
t(K_i) \leq -\chi(J_i)/2 = \extent(J_i) = \netextent(\mc{H}'_i) = \netextent_{x_i}(M_i, K_i).
\end{equation}
Observe that if $(M_i, K_i)$ is (lens space, core loop), then $t(K_i) = 0 = \netextent_{x_i}(M_i, K_i)$. 

By Theorem \ref{Net Extent Detects unknot}, we have $\netextent_{x_i}(M_i, K_i) > 0$ whenever $(M_i, K_i)$ is not $(S^3, \text{ unknot})$ or (lens space, core loop).  Combining Equation \eqref{net extent and tunnel number}, Equation \eqref{additivity application}, and inequality we obtain:
\[
t(K) \geq t(K_1) + \cdots + t(K_j) + m.
\]

Finally, a standard construction shows that $t(K) \leq t(K_1) + \hdots + t(K_n) + (n-1)$, completing the proof.
\end{proof}

For our final application, we show that higher genus bridge number, together with the genus, is super-additive under connected sum of knots that are small and m-small. A more detailed analysis would likely produce an even stronger result.

\begin{theorem}\label{Superadditivity for higher genus bridge number}
Suppose that $(\wihat{M}_i, \wihat{K}_i)$ are small and m-small for $i \in \{1, \hdots, n\}$. Let $(M,K) = \#_{i =1}^n(\wihat{M}_i, \wihat{K}_i)$ and let $g \geq g(M)$. Then there exists $g_i$ such that $\sum g_i \leq g$, $ g(\wihat{M}_i)\leq g_i$ and 
\[
\sum\limits_{i=1} (g_i + b_{g_i}(K) - 1) \leq g + b_g(K) - 1.
\]
\end{theorem}

In the following we again implicitly use the uniqueness of prime factorization.

\begin{proof}
Let $S$ be a genus $g$ bridge surface for $(M,K)$ realizing $b_g(K)$. We may perform a sequence of generalized destabilizations, undoing of removable edges, and un-perturbations to arrive at a reduced v.p.-bridge surface $H$ for $(M,K)$ with $g(H) \leq g$. Let $\mc{H} \in \H(M,K)$ be a locally thin multiple v.p.-bridge surface for $(M,K)$ such that $H \more \mc{H}$. Recall that $\netchi(\mc{H}) \leq \netchi(H)$. Let $\mc{Q}$ be the union of twice and thrice-punctured spheres in $\mc{H}^-$. By Theorem \ref{Thin summing spheres} some subset of $\mc{Q}$ is the union of summing spheres giving a prime decomposition of $(M,K)$. Let $(M', T') = (M,T)\setminus \mc{Q}$ and let $(\wihat{M}, \wihat{T})$ be the result of capping off the components of $\boundary M$ corresponding to $\mc{Q}$. Then $(\wihat{M}, \wihat{T})$ is the union of summands $(\wihat{M}_i, \wihat{K}_i)$ for $i \in \{1, \hdots, n\}$ and the union of $(S^3, \text{unknot})$ pairs. 

Suppose that $F \cpt \mc{H}^-\setminus \mc{Q}$ is contained in the interior of some $(M_i, K_i)$. By Theorem \ref{Properties Locally Thin}, $F$ is essential in $(M,K)$. If $F$ is not essential in $(\wihat{M}_i, \wihat{K}_i)$ then it must be $\boundary$-parallel in $\wihat{M}_i\setminus \eta(\wihat{K}_i)$. However, since $\wihat{K}_i$ is a knot, this implies that $F$ is a sphere intersecting $K$ twice. By the definition of $\mc{Q}$, this implies $F \cpt \mc{Q}$, a contradiction since $F$ is in the interior of $M_i$. Thus, $F$ is essential in $(\wihat{M}_i,\wihat{K}_i)$. Since $(\wihat{M}_i, \wihat{K}_i)$ is small and m-small, the surface $F$ cannot exist, and so $\mc{H}^-$ is disjoint from the interior of each $(M_i, K_i)$. 

We conclude, therefore, that each $(\wihat{M}_i, \wihat{K}_i)$ contains exactly one component $H_i$ of $\mc{H}^+$, and $H_i$ is a v.p.-bridge surface for $\wihat{M}_i$.  Let $g_i = g(H_i)$. Observe that $g_1 + \cdots + g_n \leq g$  Since $\netextent_\infty$ is non-negative for each component of $(\wihat{M}, \wihat{K})$, we have:

\[
\sum\limits_{i=1}^n g_i + |H_i \cap K_i|/2 - 1 = \sum\limits_{i=1}^n \extent(H_i) \leq \netextent(\mc{H}) \leq \extent(S) = g + b_g(K) - 1.
\]

Thus,
\[
\sum\limits_{i=1} (g_i + b_{g_i}(K) - 1) \leq g + b_g(K) - 1.
\]

as desired.
\end{proof}

Finally, we give a new proof of other theorems of Morimoto concerning composite knots which are either $(0,3)$-knots or $(1,2)$ knots. Recall that a knot $K$ in a closed 3-manifold $M$ has a $(g,b)$-decomposition if there is a genus $g$ bridge surface for $K$ intersecting $K$ in $2b$-points. The knot $K$ is a \textbf{$(g,b)$-knot} if it has a $(g,b)$-decomposition and does not have a $(g-1,b+1)$ or $(g, b-1)$ decomposition. In particular, a $(g, 0)$ knot $K$ is a core loop for a handlebody on one side of a genus $g$ Heegaard splitting of $M$. If $g = 0$, we simply say that a $(g,b)$-knot $K$ is $b$-bridge. We start with a result which may be useful in other contexts.

\begin{theorem}\label{netext 1 class}
Suppose that $K$ is a knot in a closed 3-manifold $M$ such that $(M,K)$ is irreducible and every closed surface in $M$ separates. Also assume that $(M,K)$ does not have a (lens space, core loop) connect summand. If $\netextent_{0}(M,K) = 1$, then $M$ is a $S^3$ or a lens space and $K$ has a (1,1)-decomposition. Furthermore, if $\netextent_{-2}(M,K) = 1$, then $K$ is a 2-bridge knot.
\end{theorem}
\begin{proof}
Assume, first, that $\netextent_{0}(M,K) = 1$. (If $\netextent_{-2}(M,K) = 0$, this will automatically be the case since $K$ is non-trivial and $\netextent_x$ is decreasing in $x$.) By Theorems \ref{Main Theorem} and \ref{Net Extent Detects unknot}, $(M,K)$ is prime. Let $\mc{H} \in\H(M,K)$ be locally thin with  $\netextent(\mc{H}) = \netextent_0(M,K) = 1$ and $\netchi(\mc{H}) \leq 0$. By Lemma \ref{lem:Component bound}, every component of $\mc{H}$ is a sphere or a torus. 

Recall that for any $(C, T_C) \cpt (M,T)\setminus \mc{H}$ we have
\[
\delta(C, T_C) = \extent(\boundary_+ C) - \extent(\boundary_- C)
\]
By Lemma \ref{Lem: extent difference} and Theorem \ref{Properties Locally Thin}, each $\delta(C, T_C) \geq 0$. A component of $M \setminus \mc{H}^-$ is called a \defn{chunk}. Each chunk $W$ of $M$ contains exactly two v.p.-compressionbodies of $(M,T)\setminus \mc{H}$. Let $\delta(W) = \delta(C, T_C) + \delta(D, T_D)$, where $(C, T_C)$ and $(D, T_D)$ are the v.p.-compressionbodies contained in $W$.

By equation \eqref{net extent eqn}, we have
\[
2 = 2\netextent(\mc{H}) = \sum\limits_{(C, T_C)} \delta(C, T_C) = \sum\limits_W \delta(W). \hspace{.3in} (*)
\]
The first sum is over all v.p-compressionbodies $(C, T_C) \cpt (M,T)\setminus \mc{H}$ and last sum is over all chunks $W$ of $M$. Thus, for every chunk $W$, we have $\delta(W)\leq 2$.

\textbf{Claim:} Either $M = S^3$ and $K$ is a 2-bridge knot or there is at most one v.p.-compressionbody $(C, T_C) \cpt (M,T) \setminus \mc{H}$ such that $C$ is a 3--ball. In that case, if $W$ is the chunk containing $C$, $\delta(W) = 2$ and the compressionbody adjacent to $C$ is not a product compressionbody (even when we ignore $K$).

Suppose that $(C, T_C) \cpt (M,T) \setminus \mc{H}$ is a v.p.-compressionbody with $C$ a 3--ball. Recall that by Theorem \ref{Properties Locally Thin}, $T_C \neq \nil$. Let $(D, T_D) \neq (C, T_C)$ be the v.p.-compressionbody on the other side of $\boundary_+ C$. Let $W = D \cup C$. We will show that  $\delta(W) = 2$. This implies that there is at most one such $W$. In which case, either $D$ is a 3-ball and $K$ is 2--bridge, or $C$ is the only 3-ball.

Let $b_C = |T_C|$ and observe that each component of $T_C$ is a bridge arc. Let $b_D$, $g_D$, and $v_D$ be the number of bridge arcs, ghost arcs, and vertical arcs respectively in $T_D$. Since every sphere in $M$ separates and since $K$ is a knot, $v_D$ must be even. Since $\boundary_+ D$ is a sphere, each component of $\boundary_- D$ is also a sphere. Let $\Gamma$ be the ghost arc graph for $(D, T_D)$. Since $\boundary_+ D$ is a sphere, $\Gamma$ is the union of a forest with isolated vertices. Each component of $\boundary_- D$ must intersect $K$ at least 4 times since $(M,K)$ is prime and every sphere separates $M$. Hence, each isolated vertex of $\Gamma$ intersects at least 4 vertical arcs and each leaf of $\Gamma$ intersects at least 3 vertical arcs. Let $i$ be the number of isolated vertices and $\ell$ be the number of leaves of $\Gamma$. We have, therefore, $v_D \geq 4i + 3\ell$. Observe, also that $|\boundary_- D| - g_D = \chi(\Gamma) = |\Gamma|$. 

Since the $T_D \cap \boundary_+ D = T_C \cap \boundary_+ C$, we have $b_C = b_D + v_D/2$. Thus,
\[
2 \geq \delta(W) = (-1 + b_C) + (-1 + b_D - g_D + p) \geq -2 + |\Gamma| + 2b_D + (4i + 3\ell)/2.
\]

If $|\Gamma| = 0$, then $(D, T_D)$ must be a 3--ball, and $b_D \leq 2$. Since $K$ is not the unknot, we conclude that $M = S^3$ and $K$ is a 2-bridge knot. If $|\Gamma| \geq 1$, then we have
\[
2 \geq \delta(W) \geq -1+ 2b_D + (4i + 3\ell)/2 
\]
Recall that either $i \geq 1$ or $\ell \geq 2$. Thus, $b_D = 0$. If $\ell < 2$, then $\Gamma$ is a single isolated vertex and $(D, T_D)$ is a trivial product compressionbody, a contradiction. Thus, $\ell \geq 2$ and so $2 \geq \delta(W) \geq 2$, implying $\delta(W) = 2$, as desired. This concludes the proof of the Claim.

Henceforth, assume that $(M,K)$ is not $(S^3, \text{ 2-bridge knot})$. Let $z \in \{0,1\}$ be the number of 3-balls in $M \setminus \mc{H}$. If $(A, T_A) \cpt (M,T) \setminus \mc{H}$ has the property that $A$ is a product compressionbody or a solid torus, then $\delta(A, T_A)$ is equal to the number of bridge arcs in $T_A$ (since there cannot be ghost arcs). If $A$ is a trivial product, then $\delta(A, T_A) \geq 1$ since  $(A, T_A)$ cannot be a trivial product compressionbody ($\mc{H}$ being locally thin). Thus, if $z = 1$ there are no product compressionbodies in $M \setminus \mc{H}$. 

Recall that $\netchi(\mc{H}) = -\chi(\mc{H}^+) + \chi(\mc{H}^-) \leq 0$. Every component of $\mc{H}$ is adjacent to precisely two compressionbodies, and so
\[
0 = 2\cdot 0 \geq 2\netchi(\mc{H}) = \sum\limits_{C \cpt M\setminus \mc{H}} (-\chi(\boundary_+ C) + \chi(\boundary_- C)).
\]

The only case in which $-\chi(\boundary_+ C) + \chi(\boundary_+ C) < 0$ is when $C$ is a 3-ball. Consequently,
\[
2 \geq 2z \geq \sum\limits_{C \cpt M\setminus \mc{H}} (-\chi(\boundary_+ C) + \chi(\boundary_- C)),
\]
where the sum is over all compressionbodies $C \cpt M\setminus\mc{H}$ which are not 3-balls. Since euler characteristic of a closed, orientable surface is always even, there is at most one compressionbody which is neither a product nor a solid torus nor a 3-ball. If $z = 0$, then every compressionbody must be a product or solid torus.

Suppose $z = 0$. Thus, if $W$ is an outermost chunk (i.e. a chunk with $|\boundary W| \leq 1$), one of the compressionbodies of $W \setminus \mc{H}$ is a solid torus and the other is a product or a solid torus. If both are solid tori, then $M$ is $S^3$ or a lens space and, since $\netextent(\mc{H}) = 1$, $K$ is has a (1,1)-decomposition. If one is a product, it must contain at least one bridge arc as we have previously noted. In which case, the solid torus also contains at least one bridge arc. We see, therefore, that $\delta(W) \geq 2$. Thus, $\delta(W) = 2$. Since $\boundary W \neq \nil$, there are at least two outermost chunks, $W$ and $W'$. We have $2 \geq \delta(W) + \delta(W') \geq 4$, a contradiction. 

Suppose, therefore, that $z = 1$. We will show that we again encounter a contradiction. Let $W$ be the outermost chunk containing the 3-ball $(C, T_C)$. Let $(D, T_D)$ be the other v.p.-compressionbody in $W$. We know that $\delta(W) = 2$. Since $\boundary W \neq  \nil$, there is another outermost chunk $W'$.  Let $(A, T_A)$, $(B, T_B)$ be the two v.p.-compressionbodies whose union is $W'$, with $\boundary_- A = \nil$. Since $\delta(W) = 2$, we must have $\delta(A, T_A) = \delta(B, T_B) = 0$. By Corollary \ref{delta zero}, $(A, T_A)$ must be (solid torus, core loop) or (solid torus, $\nil$). The former case can't happen since $K \cap W \neq \nil$. Both $B$ and $D$ have non-empty negative boundary. Thus, at most one of $B$ or $D$ is not a product compressionbody, which means that at least one of them is. From the claim, we know that $D$ is not a product compressionbody. But $B$  cannot be a product either, since $\delta(B, T_B) = 0$. (If it were, then $(B, T_B)$ would be a trivial product compressionbody, contradicting thinness.) Thus, this case cannot occur, either.

Finally, if we know that $\netextent_{-2}(M,K) = 1$, then we follow the same proof, but we may start with a locally thin $\mc{H}$ such that $\netchi(\mc{H}) = -2$ and $\netextent(\mc{H}) = 1$. All components of such an $\mc{H}$ must be spheres, which significantly simplifies the proof.
\end{proof}

For the statement of the next theorem, recall that each 2-bridge knot in $S^3$ has a (1,1)-decomposition.

\begin{theorem}\label{Morimoto gen}
Suppose that $K \subset S^3$ is a composite $(g,b)$ knot. Then the number of prime summands is at most $g + b - 1$. If the number of summands is exactly  $m = g+ b - 1$, then at least
\[ \frac{g}{2} + (b-1)\] of the summands have (1,1)-decompositions and at least $(b-1)$ of those are 2-bridge knots.
\end{theorem}

\begin{proof}
Suppose that $K$ is $b$-bridge with respect to a genus $g$ Heegaard surface $H$ in $S^3$. Let $x = 2g - 2$ be the euler characteristic.  By tubing along all the bridges to one side of the bridge surface (i.e. meridional stabilization) we obtain a Heegaard surface for the exterior of $K$ having genus $g + b$. Thus, $t(K) \leq g+ b - 1$. By Theorem \ref{MSS generalization} (or the earlier Scharlemann-Schultens version) $K$ can have at most $g + b - 1$ summands. Suppose it has exactly this number of summands. Let $m = g + b - 1$.

The bridge surface $H$ for $K$ is also multiple v.p.-bridge surface for $K$ with $\netextent(H) = g + b - 1$. Thus, $\netextent_x(S^3, K) \leq g + b - 1$. We apply Theorems \ref{Main Theorem} and \ref{Net Extent Detects unknot} and the uniqueness of prime decomposition. These theorems produce integers $ x_1, \hdots, x_m$, each at least -2, such that
\begin{equation}\label{xsum}
 x_1 + \cdots + x_m \leq 2g - 2 -2(m-1) = -2(b-1) 
\end{equation}
and
\[
g+b-1 \geq \netextent_x(S^3, K) = \sum\limits_{i=1}^m \netextent_{x_i}(S^3, K_i) \geq m = g+b-1.
\]
Since each $K_i$ is non-trivial, $\netextent_{x_i}(S^3, K_i) = 1$ for every $i$.

Since the $x_i$ correspond to euler characteristics of closed surfaces, they are all even. Let $n_{-}$, $n_0$, and $n_+$ be the number of $i$ such that $x_i$ is $-2$, $0$, or positive, respectively. Then Inequality \eqref{xsum} produces:
\[
-2n_{-} + 2n_+ \leq -2(b-1).
\]
Hence,
\begin{equation}\label{xsum2}
n_+ \leq n_{-} - (b-1).
\end{equation}
Additionally, since there are $m = g + b - 1$ summands, we have
\[
n_{-} + n_0 + n_+ = g + b - 1.
\]
This can be rewritten as
\begin{equation}\label{summandsum}
\big(n_{-} - (b-1)\big) + n_0 + n_+ = g.
\end{equation}

Inequality \eqref{xsum2} tells us that $n_{-} \geq (b-1)$. Combining Inequalities \eqref{xsum2} and \eqref{summandsum} produces
\[
\begin{array}{rcl}
g &\leq & \big(n_{-} - (b-1)\big) + n_0 + \big(n_{-} - (b-1)\big) \\
&\leq& 2\big(n_{-} - (b-1)\big) + 2n_0.
\end{array}
\]
Thus,
\begin{equation}\label{n0ineq}
\frac{g}{2} \leq \big(n_{-} - (b-1)\big) + n_0.
\end{equation}

Suppose, first, that $b \geq 1$. Partition the set $\{K_i : x_i \leq 0\}$ into a set $\mc{A}$ of $(b-1)$ knots $K_i$ having $x_i = -2$ and a set $\mc{B}$ of the remaining knots $K_i$ with $x_i = - 2$ and the knots with $x_i = 0$. By Theorem \ref{netext 1 class}, each knot in $\mc{A}$ is a 2-bridge knot and each knot in $\mc{B}$ has a (1,1)-decomposition. The result follows from the fact that $|\mc{A} \cup \mc{B}| = |\mc{A}| + |\mc{B}| \geq (b-1) + \frac{g}{2}$.

Now suppose that $b = 0$. It is then trivially true that at least $b-1$ summands $K_i$ have $x_i = -2$. Let $\mc{A} = \nil$ and let $\mc{B}$ be the set of all $K_i$ with $x_i \in \{-2, 0\}$. By Inequality \eqref{n0ineq}, we have
\[
\frac{g}{2} + b - 1\leq  |\mc{B}|
\]
and the result follows as before.
\end{proof}

Applying the theorem with $m = 2$, produces the aforementioned results of Morimoto.

\begin{corollary}[{Morimoto \cite[Theorems 3, 4]{Morimoto15}}]\label{MorimotoCor}
Suppose that $K \subset S^3$ is a composite knot which is either 3-bridge with respect to a sphere or 2-bridge with respect to a Heegaard torus for $S^3$. Then, in the former case, $K$ is the connected sum of two 2-bridge knots and in the latter case it is a connected sum of a 2-bridge knot and a (1,1)-knot.
\end{corollary}
\begin{proof}
We apply Theorem \ref{Morimoto gen} with $(g,b)$ either $(0,3)$ or $(1,2)$. The quantity $g + b - 1 = 2$, so $K$ has at most two prime factors. At least $\frac{g}{2} + (b-1)$ of those factors have (1,1)-decompositions and $(b-1)$ of those are 2-bridge. If $g = 0$, then $b - 1 = 2$ and the result follows. If $g = 1$, then $g/2 + (b-1) = 3/2$. Since there are two summands, both must have (1,1)-decompositions and at least one must be a 2-bridge knot. If both were 2-bridge knots, then $K$ would be a $(0,3)$-knot by Schubert's theorem.
\end{proof}

\subsection*{Acknowledgements} The first author was supported by a grant from the Colby College Division of Natural Sciences. The second named author is grateful to the NSF for supporting this work through a CAREER grant. The authors also thank Ryan Blair for helpful conversations regarding Section \ref{comparison}.

\end{document}